\documentclass[12pt,a4paper]{article}
\usepackage[utf8]{inputenc}
\usepackage{pgf,tikz}
\usepackage{tikz-cd}
\usepackage{mathrsfs}
\usetikzlibrary{arrows}
\usepackage[T1]{fontenc}
\usepackage{amssymb}
\usepackage{amsmath}
\usepackage{amsthm}
\usepackage{mathcomp}
\usepackage{fancyhdr}
\usepackage{geometry}
\usepackage{caption}
\usepackage{subcaption}
\usepackage{titling}
\usepackage{stmaryrd}
\usepackage{array,multirow,makecell}
\usepackage{soul}
\usepackage{xcolor}
\usepackage{color}
\usepackage{framed}
\usepackage{algorithm}
\usepackage{algpseudocode}
\usepackage{listings}
\usepackage{bbm}
\usepackage{bm}
\usepackage{tikz}
\usetikzlibrary{automata, arrows.meta, positioning}
\usetikzlibrary{patterns}
\setcellgapes{1pt}
\makegapedcells
\newcolumntype{R}[1]{>{\raggedleft\arraybackslash }b{#1}}
\newcolumntype{L}[1]{>{\raggedright\arraybackslash }b{#1}}
\newcolumntype{C}[1]{>{\centering\arraybackslash }b{#1}}
\makeatletter
\def\overparenthesis#1{\mathop{\vbox{\ialign{##\crcr\noalign{\kern3\p@}\downparenthfill\crcr\noalign{\kern3\p@\nointerlineskip}$\hfil\displaystyle{#1}\hfil$\crcr}}}\limits}
\def\underparenthesis#1{\mathop{\vtop{\ialign{##\crcr$\hfil\displaystyle{#1}\hfil$\crcr\noalign{\kern3\p@\nointerlineskip}\upparenthfill\crcr\noalign{\kern3\p@}}}}\limits}
\def\downparenthfill{$\m@th\braceld\leaders\vrule\hfill\bracerd$}
\def\upparenthfill{$\m@th\bracelu\leaders\vrule\hfill\braceru$}
\makeatother

\usepackage{hyperref}

\newtheorem{theorem}{Theorem}
\newtheorem{lemma}[theorem]{Lemma}

\newtheorem{definition}[theorem]{Definition}
\newtheorem{corollary}[theorem]{Corollary}

\newtheorem{proposition}[theorem]{Proposition}

\newcommand{\ZZ}{\ensuremath{\mathbb{Z}}}

\newcommand{\NN}{\ensuremath{\mathbb{N}}}

\newcommand{\limset}{\Omega}
\newcommand{\gls}{\Tilde{\omega}}
\DeclareMathOperator{\supp}{supp}
\newcommand{\meas}{\mathcal{M}}
\newcommand{\borel}{\mathcal{B}}



\title{Limit dynamics of elementary cellular automaton 18}

\author{Hervé Sabrié \\ Computer Science Department \\ ENS Paris Saclay \\ Paris, France \and Ilkka Törmä \\ Department of Mathematics and Statistics \\ University of Turku \\ Turku, Finland \\ \href{mailto:iatorm@utu.fi}{iatorm@utu.fi}}

\begin{document}
\maketitle

\begin{abstract}
    We study the the asymptotic dynamics of elementary cellular automaton 18 through its limit set, generic limit set and $\mu$-limit set.
    The dynamics of rule 18 are characterized by persistent local patterns known as kinks.
    We characterize the configurations of the generic limit set containing at most two kinks.
    As a corollary, we show that the three limit sets of rule 18 are distinct.
\end{abstract}

\subsection*{Acknowledgements}

Ilkka Törmä was supported by the Academy of Finland under grant 346566.

\section{Introduction}

Cellular automata (CA) are a class of discrete dynamical systems consisting of an infinite line (or more generally a multidimensional grid) of ``cells'', each of which holds one of finitely many states.
Each cell evolves by an identical rule, and its next state is determined by the current state of itself and some finite set of neighboring cells.
The 256 \emph{elementary} cellular automata \cite{Wo83} are those with only two states and neighborhood $N = \{-1,0,1\}$, meaning that each cell can communicate only with its nearest neighbors.
They are arguably the simplest cellular automata that exhibit complex dynamics.

We focus on the dynamics of elementary CA number 18.
It has been studied as a model of mutually annihilating particles with random movement as early as 1983 \cite{GRA83}, and by several authors since \cite{JEN90,KUR03}.
Grassberger \cite{GRA83} and Lind \cite{LIN84} conjectured that if the CA is initialized from a uniformly random configuration, then the density of particles almost surely decreases to 0 with the same speed as mutually annihilating independent random walks.
Progress toward the conjecture has been slow.
K\r{u}rka claimed in \cite{KUR03} that the density indeed approaches 0, but we show that the presented proof is faulty (see the discussion after Proposition \ref{prop:stable-extensions-f}).

The long-term fate of a dynamical system can be studied via its \emph{limit set}, or eventual image.
The limit set of a cellular automaton often contains transient phenomena, so several variants have been proposed to better capture the intuitive notion of ``what typically happens after a long time''.
These include the \emph{$\mu$-limit set} \cite{KM00} and the \emph{generic limit set}, introduced by Milnor in \cite{MIL85} and studied first in the context of cellular automata in \cite{DJEGUI19}.
As our main results, we characterize the words with at most two particles that occur in the generic limit set of rule 18, and show that the three limit sets are all distinct.
We also show that determining the generic limit set of rule 18 is a viable strategy toward (partially) resolving the conjecture of Grassberger and Lind.

\section{Definitions}

Let $A$ be a finite alphabet.
The set of finite words over $A$ is denoted by $A^*$, and concatenation of words $v, w \in A^*$ is denoted by either $v w$ or $v \cdot w$.
This notation is extended to infinite words: if $v \in A^*$ and $x \in A^\NN$ is a right-infinite word, then $v x \in A^\NN$; if $y \in A^{-\NN}$ is a left-infinite word, then $y v \in A^{-\NN}$; and $y x \in A^\ZZ$ is a bi-infinite word.
The \emph{full shift} is the set $A^\ZZ$ of all bi-infinite sequences over $A$, equipped with the prodiscrete topology, which is generated by the \emph{cylinder sets} $[w]_i = \{x \in A^\ZZ \mid x_{[i, i+|w|-1]} = w\}$ for $w \in A^*$ and $i \in \ZZ$.
The \emph{shift map} $\sigma : A^\ZZ \to A^\ZZ$ is defined by $\sigma(x)_i = x_{i+1}$.
A \emph{subshift} is a topologically closed set $X \subseteq A^\ZZ$ that is shift-invariant: $\sigma(X) = X$.
A word $w \in A^*$ \emph{occurs} in $X$ if $X \cap [w]_0 \neq \emptyset$.
The \emph{language} of $X$ is the set of words occurring in it; a subshift is defined by its language.

A \emph{cellular automaton} (CA) is a continuous function $f : A^\ZZ \to A^\ZZ$ that commutes with the shift: $f \circ \sigma = \sigma \circ f$.
By the Curtis-Hedlund-Lyndon theorem \cite{HED69}, every CA has a finite \emph{radius} $r \geq 0$ and a \emph{local rule} $F : A^{2r+1} \to A$ such that $f(x)_i = F(\sigma^i(x)|_{[-r,r]})$ for all $x \in A^\ZZ$ and $i \in \ZZ$.
Once $r$ is fixed, we can apply the CA to words $w \in A^n$ as well as configurations: $f(w) = F(w_{[0,2r]} w_{[1,2r+1]} \cdots w_{[n-2r-1,n-1]})$.
Note that each application of $f$ shortens the word by $2r$ symbols.

The \emph{elementary cellular automata} (ECA), defined in \cite{Wo83}, are those with alphabet $A = \{0,1\}$ and radius $1$.
They are numbered from 0 to 255: the local rule $F$ of ECA $n$ maps $a b c \in \{0,1\}^3$ to bit number $4a + 2b + c$ in the binary representation of $n$ (extended to 3 digits by prepending 0s if necessary).

A \emph{topological dynamical system} is a pair $(X, T)$, where $X$ is a compact metric space and $T : X \to X$ is a continuous function.
Cellular automata are examples of such systems.


A \emph{Borel probability measure} on a topological space $X$ is a function $\mu : \borel(X) \to [0,1]$, where $\borel(X)$ is the class of Borel subsets of $X$, satisfying $\mu(X) = 1$, $\mu(\emptyset) = 0$ and $\mu(\bigcup_{n \in \NN} X_n) = \sum_{n \in \NN} \mu(X_n)$ whenever the sets $X_n$ are pairwise disjoint.
The set of Borel probability measures of $X$ is denoted $\meas(X)$.
The \emph{support} of $\mu \in \meas(X)$ is the smallest closed subset $\supp(\mu) \subseteq X$ with measure $1$.
If $(X,T)$ is a dynamical system, the dynamics $T$ can be extended to a function $T : \meas(X) \to \meas(X)$ by $(T \mu)(K) = \mu(T^{-1}(K))$.
We denote by $\meas_T(X)$ the subset of measures that are $T$-invariant, meaning $T \mu = \mu$.
A measure $\mu \in \meas_T(X)$ is \emph{$T$-ergodic} if $T(A) = A$ for a Borel set $A \subseteq X$ implies $\mu(A) \in \{0,1\}$.

We are mostly interested in shift-invariant measures on $A^\ZZ$, which are characterized by their values on cylinders of the form $[w]_0$.
For convenience, we denote $\mu(w) = \mu([w]_0)$.
The \emph{uniform Bernoulli measure} is given by $\mu(w) = |A|^{-|w|}$ for all $w \in A^*$, and it is shift-ergodic.
It corresponds to the random process where every symbol of an infinite configuration is chosen independently and uniformly at random.
The set $\meas(X)$ can be endowed with several topologies; we use the weak topology, which for the full shift is equivalent to convergence on cylinder sets: $(\mu_n)$ weakly converges to $\mu$ if and only if $\mu_n([w]_i) \to \mu([w]_i)$ for all cylinders $[w]_i$.


\section{Limit sets and variants}

We now define the limit sets of a dynamical system that will be used in this paper.

\begin{definition}[\cite{MIL85,KM00}]
Let $(X, T)$ be a topological dynamical system.
\begin{itemize}
    \item The \emph{limit set} of $T$ is $\limset(T) = \bigcap_{n \in \NN} T^n(X)$.
    \item The \emph{realm of attraction} of a subset $K \subset X$ is the set
    \[ D_T(K) = \{ x \in X \mid \text{all limit points of $(T^n(x))_{n \in \NN}$ are in $K$} \}. \]
    The \emph{generic limit set} of $T$ is the intersection of all closed subsets $K \subset X$ such that $D_T(K)$ is comeager.
    \item Let $\mu$ be a Borel probability measure on $X$.
    The \emph{$\mu$-limit set} of $T$ is the set $\limset_\mu(T) = \overline{\bigcup_\nu \supp(\nu)}$, where $\nu$ ranges over the weak limit points of the sequence $(T^n \mu)_{n \in \NN}$.
\end{itemize}
\end{definition}

The limit set can also be called the \emph{eventual image}, and it is purely set-theoretic.
The generic limit set is its topological refinement that aims to capture the asymptotic dynamics of the system from a generic (in a topological sense) initial condition.
The $\mu$-limit set captures those phenomena whose probability of occurrence does not vanish in the limit.

If $f : A^\ZZ \to A^\ZZ$ is a cellular automaton and $\mu$ is a shift-invariant measure on $A^\ZZ$, then a given word $w$ occurs in the $\mu$-limit set $\limset_\mu(f)$ if and only if $\mu(f^{-n}(w))$ does not converge to $0$ as $n$ grows.
In particular, if $\mu$ is the uniform Bernoulli measure, then $w$ occurs in $\limset_\mu(f)$ if and only if $\limsup_n |f^{-n}(w)| / |A|^{|w|+2nr} > 0$.
Generic limit sets of cellular automata have a combinatorial characterization as well.

\begin{lemma}[\cite{TOR20}]
\label{lemma:combchar}
Let $f : A^\ZZ \to A^\ZZ$ be a cellular automaton and $w \in A^*$.
Then $w$ occurs in the generic limit set $\gls(f)$ if and only if there exists a ``seed'' word $s \in A^*$ and a position $i \in \ZZ$ such that for all $u, v \in A^*$, there are infinitely many $n \in \NN$ with $f^n([u s v]_{i - |u|}) \cap [w]_0 \neq \emptyset$.
\end{lemma}

All three limit sets are nonempty subshifts of $A^\ZZ$.
The limit set contains the other two, but in general neither of the generic limit set or the $\mu$-limit set contains the other.

\begin{lemma}[Proposition 6.2 of \cite{DJEGUI19}]
\label{lem:three-inclusion}
Let $f : A^\ZZ \to A^\ZZ$ be a cellular automaton.
Then $\gls(f) \subseteq \limset(f)$ and $\limset_\mu(f) \subseteq \limset(f)$.
\end{lemma}

\section{Defect dynamics of rule 18}

The local rule of ECA 18 is shown in Table \ref{tab:rule18}: the patterns 001 and 100 produce a 1, and all other patterns produce a 0.
In the remainder of this paper, we denote this CA by $f_{18}$, or just $f$ when there is no danger of confusion.
Figure \ref{fig:example-run} depicts a sample spacetime configuration of $f$.
In this figure and in the following, time advances downward, a black square represents a 1 and a white square represent a 0.
Hatched squares are undetermined symbols.

\begin{table}[htp]
    \centering
    \begin{tabular}{cccccccc}
        000 & 001 & 010 & 011 & 100 & 101 & 110 & 111 \\
         0  &  1  &  0  &  0  &  1  &  0  &  0  &  0 
    \end{tabular}
    \caption{The local rule of ECA 18.}
    \label{tab:rule18}
\end{table}

\begin{figure}[htp]
\centering
\begin{tikzpicture}[scale=0.22]
\draw[step=1cm, gray, very thin] (-0.1,-0.1) grid (45.1,18.1);
\fill[black!80] (4,0) rectangle ++(1,1);
\fill[black!80] (14,0) rectangle ++(1,1);
\fill[black!80] (16,0) rectangle ++(1,1);
\fill[black!80] (20,0) rectangle ++(1,1);
\fill[black!80] (24,0) rectangle ++(1,1);
\fill[black!80] (32,0) rectangle ++(1,1);
\fill[black!80] (34,0) rectangle ++(1,1);
\fill[black!80] (36,0) rectangle ++(1,1);
\fill[black!80] (38,0) rectangle ++(1,1);
\fill[black!80] (40,0) rectangle ++(1,1);
\fill[black!80] (42,0) rectangle ++(1,1);
\fill[black!80] (5,1) rectangle ++(1,1);
\fill[black!80] (6,1) rectangle ++(1,1);
\fill[black!80] (8,1) rectangle ++(1,1);
\fill[black!80] (10,1) rectangle ++(1,1);
\fill[black!80] (12,1) rectangle ++(1,1);
\fill[black!80] (13,1) rectangle ++(1,1);
\fill[black!80] (17,1) rectangle ++(1,1);
\fill[black!80] (19,1) rectangle ++(1,1);
\fill[black!80] (25,1) rectangle ++(1,1);
\fill[black!80] (27,1) rectangle ++(1,1);
\fill[black!80] (29,1) rectangle ++(1,1);
\fill[black!80] (31,1) rectangle ++(1,1);
\fill[black!80] (35,1) rectangle ++(1,1);
\fill[black!80] (39,1) rectangle ++(1,1);
\fill[black!80] (43,1) rectangle ++(1,1);
\fill[black!80] (44,1) rectangle ++(1,1);
\fill[black!80] (0,2) rectangle ++(1,1);
\fill[black!80] (2,2) rectangle ++(1,1);
\fill[black!80] (4,2) rectangle ++(1,1);
\fill[black!80] (7,2) rectangle ++(1,1);
\fill[black!80] (11,2) rectangle ++(1,1);
\fill[black!80] (14,2) rectangle ++(1,1);
\fill[black!80] (16,2) rectangle ++(1,1);
\fill[black!80] (20,2) rectangle ++(1,1);
\fill[black!80] (22,2) rectangle ++(1,1);
\fill[black!80] (24,2) rectangle ++(1,1);
\fill[black!80] (28,2) rectangle ++(1,1);
\fill[black!80] (32,2) rectangle ++(1,1);
\fill[black!80] (34,2) rectangle ++(1,1);
\fill[black!80] (40,2) rectangle ++(1,1);
\fill[black!80] (42,2) rectangle ++(1,1);
\fill[black!80] (3,3) rectangle ++(1,1);
\fill[black!80] (8,3) rectangle ++(1,1);
\fill[black!80] (10,3) rectangle ++(1,1);
\fill[black!80] (15,3) rectangle ++(1,1);
\fill[black!80] (21,3) rectangle ++(1,1);
\fill[black!80] (25,3) rectangle ++(1,1);
\fill[black!80] (27,3) rectangle ++(1,1);
\fill[black!80] (33,3) rectangle ++(1,1);
\fill[black!80] (41,3) rectangle ++(1,1);
\fill[black!80] (0,4) rectangle ++(1,1);
\fill[black!80] (2,4) rectangle ++(1,1);
\fill[black!80] (9,4) rectangle ++(1,1);
\fill[black!80] (16,4) rectangle ++(1,1);
\fill[black!80] (18,4) rectangle ++(1,1);
\fill[black!80] (20,4) rectangle ++(1,1);
\fill[black!80] (26,4) rectangle ++(1,1);
\fill[black!80] (34,4) rectangle ++(1,1);
\fill[black!80] (36,4) rectangle ++(1,1);
\fill[black!80] (38,4) rectangle ++(1,1);
\fill[black!80] (40,4) rectangle ++(1,1);
\fill[black!80] (3,5) rectangle ++(1,1);
\fill[black!80] (5,5) rectangle ++(1,1);
\fill[black!80] (7,5) rectangle ++(1,1);
\fill[black!80] (8,5) rectangle ++(1,1);
\fill[black!80] (17,5) rectangle ++(1,1);
\fill[black!80] (21,5) rectangle ++(1,1);
\fill[black!80] (23,5) rectangle ++(1,1);
\fill[black!80] (25,5) rectangle ++(1,1);
\fill[black!80] (35,5) rectangle ++(1,1);
\fill[black!80] (39,5) rectangle ++(1,1);
\fill[black!80] (0,6) rectangle ++(1,1);
\fill[black!80] (2,6) rectangle ++(1,1);
\fill[black!80] (6,6) rectangle ++(1,1);
\fill[black!80] (9,6) rectangle ++(1,1);
\fill[black!80] (11,6) rectangle ++(1,1);
\fill[black!80] (13,6) rectangle ++(1,1);
\fill[black!80] (14,6) rectangle ++(1,1);
\fill[black!80] (16,6) rectangle ++(1,1);
\fill[black!80] (22,6) rectangle ++(1,1);
\fill[black!80] (26,6) rectangle ++(1,1);
\fill[black!80] (28,6) rectangle ++(1,1);
\fill[black!80] (30,6) rectangle ++(1,1);
\fill[black!80] (32,6) rectangle ++(1,1);
\fill[black!80] (34,6) rectangle ++(1,1);
\fill[black!80] (40,6) rectangle ++(1,1);
\fill[black!80] (42,6) rectangle ++(1,1);
\fill[black!80] (43,6) rectangle ++(1,1);
\fill[black!80] (1,7) rectangle ++(1,1);
\fill[black!80] (7,7) rectangle ++(1,1);
\fill[black!80] (8,7) rectangle ++(1,1);
\fill[black!80] (12,7) rectangle ++(1,1);
\fill[black!80] (15,7) rectangle ++(1,1);
\fill[black!80] (23,7) rectangle ++(1,1);
\fill[black!80] (25,7) rectangle ++(1,1);
\fill[black!80] (29,7) rectangle ++(1,1);
\fill[black!80] (33,7) rectangle ++(1,1);
\fill[black!80] (41,7) rectangle ++(1,1);
\fill[black!80] (44,7) rectangle ++(1,1);
\fill[black!80] (2,8) rectangle ++(1,1);
\fill[black!80] (4,8) rectangle ++(1,1);
\fill[black!80] (6,8) rectangle ++(1,1);
\fill[black!80] (9,8) rectangle ++(1,1);
\fill[black!80] (11,8) rectangle ++(1,1);
\fill[black!80] (16,8) rectangle ++(1,1);
\fill[black!80] (18,8) rectangle ++(1,1);
\fill[black!80] (20,8) rectangle ++(1,1);
\fill[black!80] (22,8) rectangle ++(1,1);
\fill[black!80] (26,8) rectangle ++(1,1);
\fill[black!80] (28,8) rectangle ++(1,1);
\fill[black!80] (34,8) rectangle ++(1,1);
\fill[black!80] (36,8) rectangle ++(1,1);
\fill[black!80] (38,8) rectangle ++(1,1);
\fill[black!80] (40,8) rectangle ++(1,1);
\fill[black!80] (1,9) rectangle ++(1,1);
\fill[black!80] (5,9) rectangle ++(1,1);
\fill[black!80] (10,9) rectangle ++(1,1);
\fill[black!80] (17,9) rectangle ++(1,1);
\fill[black!80] (21,9) rectangle ++(1,1);
\fill[black!80] (27,9) rectangle ++(1,1);
\fill[black!80] (35,9) rectangle ++(1,1);
\fill[black!80] (39,9) rectangle ++(1,1);
\fill[black!80] (0,10) rectangle ++(1,1);
\fill[black!80] (6,10) rectangle ++(1,1);
\fill[black!80] (7,10) rectangle ++(1,1);
\fill[black!80] (9,10) rectangle ++(1,1);
\fill[black!80] (18,10) rectangle ++(1,1);
\fill[black!80] (20,10) rectangle ++(1,1);
\fill[black!80] (28,10) rectangle ++(1,1);
\fill[black!80] (30,10) rectangle ++(1,1);
\fill[black!80] (32,10) rectangle ++(1,1);
\fill[black!80] (34,10) rectangle ++(1,1);
\fill[black!80] (40,10) rectangle ++(1,1);
\fill[black!80] (42,10) rectangle ++(1,1);
\fill[black!80] (43,10) rectangle ++(1,1);
\fill[black!80] (1,11) rectangle ++(1,1);
\fill[black!80] (3,11) rectangle ++(1,1);
\fill[black!80] (5,11) rectangle ++(1,1);
\fill[black!80] (8,11) rectangle ++(1,1);
\fill[black!80] (19,11) rectangle ++(1,1);
\fill[black!80] (29,11) rectangle ++(1,1);
\fill[black!80] (33,11) rectangle ++(1,1);
\fill[black!80] (41,11) rectangle ++(1,1);
\fill[black!80] (44,11) rectangle ++(1,1);
\fill[black!80] (2,12) rectangle ++(1,1);
\fill[black!80] (6,12) rectangle ++(1,1);
\fill[black!80] (7,12) rectangle ++(1,1);
\fill[black!80] (20,12) rectangle ++(1,1);
\fill[black!80] (21,12) rectangle ++(1,1);
\fill[black!80] (23,12) rectangle ++(1,1);
\fill[black!80] (25,12) rectangle ++(1,1);
\fill[black!80] (27,12) rectangle ++(1,1);
\fill[black!80] (28,12) rectangle ++(1,1);
\fill[black!80] (34,12) rectangle ++(1,1);
\fill[black!80] (36,12) rectangle ++(1,1);
\fill[black!80] (38,12) rectangle ++(1,1);
\fill[black!80] (40,12) rectangle ++(1,1);
\fill[black!80] (3,13) rectangle ++(1,1);
\fill[black!80] (5,13) rectangle ++(1,1);
\fill[black!80] (8,13) rectangle ++(1,1);
\fill[black!80] (10,13) rectangle ++(1,1);
\fill[black!80] (12,13) rectangle ++(1,1);
\fill[black!80] (14,13) rectangle ++(1,1);
\fill[black!80] (15,13) rectangle ++(1,1);
\fill[black!80] (17,13) rectangle ++(1,1);
\fill[black!80] (19,13) rectangle ++(1,1);
\fill[black!80] (22,13) rectangle ++(1,1);
\fill[black!80] (26,13) rectangle ++(1,1);
\fill[black!80] (29,13) rectangle ++(1,1);
\fill[black!80] (31,13) rectangle ++(1,1);
\fill[black!80] (33,13) rectangle ++(1,1);
\fill[black!80] (37,13) rectangle ++(1,1);
\fill[black!80] (41,13) rectangle ++(1,1);
\fill[black!80] (42,13) rectangle ++(1,1);
\fill[black!80] (44,13) rectangle ++(1,1);
\fill[black!80] (4,14) rectangle ++(1,1);
\fill[black!80] (9,14) rectangle ++(1,1);
\fill[black!80] (13,14) rectangle ++(1,1);
\fill[black!80] (16,14) rectangle ++(1,1);
\fill[black!80] (20,14) rectangle ++(1,1);
\fill[black!80] (21,14) rectangle ++(1,1);
\fill[black!80] (27,14) rectangle ++(1,1);
\fill[black!80] (28,14) rectangle ++(1,1);
\fill[black!80] (32,14) rectangle ++(1,1);
\fill[black!80] (38,14) rectangle ++(1,1);
\fill[black!80] (40,14) rectangle ++(1,1);
\fill[black!80] (43,14) rectangle ++(1,1);
\fill[black!80] (1,15) rectangle ++(1,1);
\fill[black!80] (3,15) rectangle ++(1,1);
\fill[black!80] (10,15) rectangle ++(1,1);
\fill[black!80] (12,15) rectangle ++(1,1);
\fill[black!80] (17,15) rectangle ++(1,1);
\fill[black!80] (19,15) rectangle ++(1,1);
\fill[black!80] (22,15) rectangle ++(1,1);
\fill[black!80] (24,15) rectangle ++(1,1);
\fill[black!80] (26,15) rectangle ++(1,1);
\fill[black!80] (29,15) rectangle ++(1,1);
\fill[black!80] (31,15) rectangle ++(1,1);
\fill[black!80] (39,15) rectangle ++(1,1);
\fill[black!80] (44,15) rectangle ++(1,1);
\fill[black!80] (2,16) rectangle ++(1,1);
\fill[black!80] (11,16) rectangle ++(1,1);
\fill[black!80] (18,16) rectangle ++(1,1);
\fill[black!80] (23,16) rectangle ++(1,1);
\fill[black!80] (27,16) rectangle ++(1,1);
\fill[black!80] (28,16) rectangle ++(1,1);
\fill[black!80] (32,16) rectangle ++(1,1);
\fill[black!80] (34,16) rectangle ++(1,1);
\fill[black!80] (36,16) rectangle ++(1,1);
\fill[black!80] (38,16) rectangle ++(1,1);
\fill[black!80] (0,17) rectangle ++(1,1);
\fill[black!80] (1,17) rectangle ++(1,1);
\fill[black!80] (12,17) rectangle ++(1,1);
\fill[black!80] (13,17) rectangle ++(1,1);
\fill[black!80] (15,17) rectangle ++(1,1);
\fill[black!80] (16,17) rectangle ++(1,1);
\fill[black!80] (17,17) rectangle ++(1,1);
\fill[black!80] (24,17) rectangle ++(1,1);
\fill[black!80] (25,17) rectangle ++(1,1);
\fill[black!80] (26,17) rectangle ++(1,1);
\fill[black!80] (29,17) rectangle ++(1,1);
\fill[black!80] (30,17) rectangle ++(1,1);
\fill[black!80] (31,17) rectangle ++(1,1);
\fill[black!80] (35,17) rectangle ++(1,1);
\fill[black!80] (39,17) rectangle ++(1,1);
\fill[black!80] (40,17) rectangle ++(1,1);
\fill[black!80] (42,17) rectangle ++(1,1);
\fill[black!80] (43,17) rectangle ++(1,1);
\fill[black!80] (44,17) rectangle ++(1,1);
\end{tikzpicture} 
\caption{An example run of rule 18.}
\label{fig:example-run}
\end{figure}
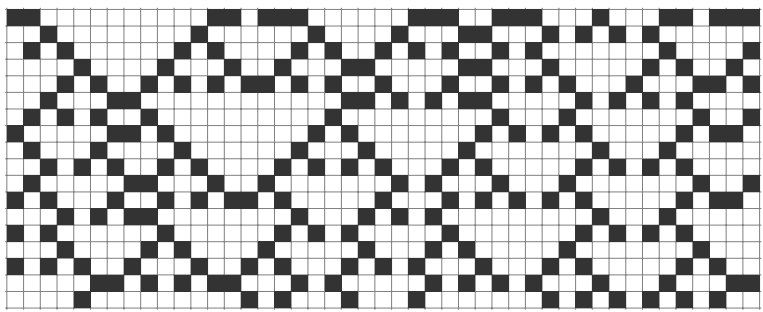

The dynamics of ECA rule 18 are a combination of several interacting processes.
On regions where the gaps between 1-symbols have odd length, the dynamics is linear and well understood.
A single gap of even length, called \emph{kink}, is persistent.
Its movements are influenced by the linear part of the dynamics, but not vice versa: the linear dynamics commutes with a local operation that erases the kink.
However, when two kinks collide, both are destroyed and the linear dynamics is perturbed.



\begin{definition}
Words of the form $1 0^{2n} 1$ with $n \in \NN$ are called \emph{kinks}.
The \emph{left border} of the kink is its leftmost $1$-symbol.

Let $x \in \{0,1\}^{\ZZ}$.
We denote $g_x(n) = \#\{-n \leq i \leq n \mid \exists k \in \NN, x_{[i, i+2k+2]} = 1 0^{2k} 1\}$.
If $\lim_{n \rightarrow \infty} g_x(n) = m < \infty$, then we call $m$ the number of kinks in $x$.

We call $\Sigma \subseteq \{0,1\}^*$ the set of words without kinks, and $\Tilde \Sigma \subseteq \{0,1\}^\ZZ$ the set of configurations without kinks.
\end{definition}

\begin{proposition}[\cite{JEN90}]
\label{prop:rule90}
Let $f_{90}$ be the function associated to ECA rule $90$: $f_{90}(abc) = a \oplus c$ for all $abc \in \{0,1\}^3$, where ${\oplus}$ is addition modulo 2, or the xor function. Then, for all kinkless words $w \in \Sigma$ we have $f_{18}(w) = f_{90}(w)$, and for all kinkless configurations $x \in \Tilde \Sigma$ we have $f_{18}(x) = f_{90}(x)$.
\end{proposition}



New kinks cannot be created by rule 18.
Furthermore, in ``generic'' infinite configurations, kinks can only be destroyed in pairs.
Consider a word of the form $u = 0 0 1 w 1 0 0$, where $w$ does not contain $00$ as a subword.
Then $f_{18}(u) = 1 0^{|w|+2} 1$, and this word is a kink if and only if the number of kinks in $u$ is odd, as illustrated in Figure \ref{fig:kink-elim}.
As a limiting case, consider a configuration $x \in \{0,1\}^\ZZ$ such that some infinite tail $x_{[i, \infty)}$ contains no occurrences of $00$.
Then $f(x)$ contains the same number of kinks as $f(y)$, where $y$ is obtained from $x$ by replacing the infinite tail by one of $(01)^\infty$ or $(10)^\infty$ in a way that does not introduce a new $00$.
These cases (and the symmetric one where we consider an infinite left tail) are the only ways to erase kinks in infinite configurations.

When $f_{18}$ is applied to finite words, their length decreases, so it is inevitable that all kinks are either destroyed as described above, or ``fall off the edge of the word'' as in $f(1100) = 01$.
The notion of stable words introduced later will allow us to control this phenomenon.

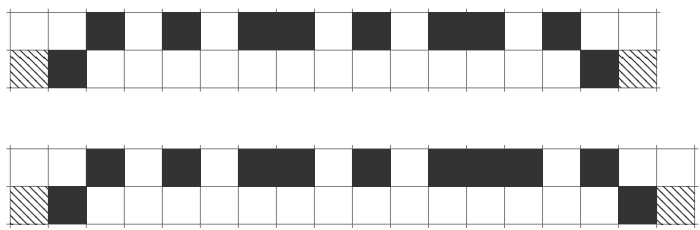
\begin{figure}[htp]
\centering
    \begin{tabular}{c}
    \begin{subfigure}[b]{.5\linewidth}
    \begin{tikzpicture}[scale = 0.5]
\draw[step=1cm, gray, very thin](-0.1,-0.1)grid(17.1,2.1);
\fill[black!80](2,1)rectangle(3,2);
\fill[black!80](4,1)rectangle(5,2);
\fill[black!80](6,1)rectangle(8,2);
\fill[black!80](9,1)rectangle(10,2);
\fill[black!80](11,1)rectangle(13,2);
\fill[black!80](14,1)rectangle(15,2);
\fill[black!80](1,0)rectangle(2,1);
\fill[black!80](15,0)rectangle(16,1);
\fill[pattern=north west lines, pattern color=black!80] (0,0) rectangle (1,1);
\fill[pattern=north west lines, pattern color=black!80] (16,0) rectangle (17,1);
\end{tikzpicture}
    \end{subfigure}
    \\
    \begin{subfigure}[b]{.5\linewidth}
    \begin{tikzpicture}[scale = 0.5]
\draw[step=1cm, gray, very thin](-0.1,-0.1)grid(18.1,2.1);
\fill[black!80](2,1)rectangle(3,2);
\fill[black!80](4,1)rectangle(5,2);
\fill[black!80](6,1)rectangle(8,2);
\fill[black!80](9,1)rectangle(10,2);
\fill[black!80](11,1)rectangle(14,2);
\fill[black!80](15,1)rectangle(16,2);
\fill[black!80](1,0)rectangle(2,1);
\fill[black!80](16,0)rectangle(17,1);
\fill[pattern=north west lines, pattern color=black!80] (0,0) rectangle (1,1);
\fill[pattern=north west lines, pattern color=black!80] (17,0) rectangle (18,1);
\end{tikzpicture}
    \end{subfigure}
    \end{tabular}
    \caption{
    The word $u_1$ above contains an even number of kinks, hence $f(u_1)$ is not a kink. The word $u_2$ below contains an odd number of kinks, hence $f(u_2)$ is a kink.}
    \label{fig:kink-elim}
\end{figure}


The following definitions are inspired by \cite{KUR03}.

\begin{definition}
\label{def:extensions}
Let $w \in \{0,1\}^*$ be a word that contains $m \in \NN$ kinks.
The \emph{finite extension} $\Sigma(w) \subseteq \{0,1\}^*$ of $w$ is the set of those words that contains $w$ as a subword and contain exactly $m$ kinks.
The \emph{extension} of $w$ at position $p \in \ZZ$ is $\Tilde \Sigma(w, p) = \{x \in [w]_p \mid \text{$x$ contains $m$ kinks} \}$.
\end{definition}

\begin{definition}
The (regular) language of \emph{left unstable words} is $(\epsilon + 0) (10)^* 11 (0+1)^*$, and the language of \emph{right unstable words} is $(0+1)^* 11 (01)^* (\epsilon + 0)$.
A word is \emph{unstable} if it is left or right unstable.
If a word is not unstable, it is \emph{stable}.
We denote by $S$ the set of stable words.
\end{definition}

Extensions of stable words are also stable; we omit the easy proof.

\begin{proposition}
Let $w \in S$ be a stable word.
Then $\Sigma(w) \subseteq S$. 
\end{proposition}

In stable words, kinks are necessarily erased in pairs.
This is depicted in Figure \ref{fig:kink-elim}.
From this and the discussion preceding Definition \ref{def:extensions} we infer the following.

\begin{proposition}
\label{prop:no-new-kinks-stable}
Let $w \in \{0,1\}^*$ be a word that contains $n \in \NN$ kinks.
Then $f(w)$ contains $m \leq n$ kinks.
If $w \in S$, then $m \equiv n \bmod 2$.
\end{proposition}

The most important property of stable words is that they behave nicely with respect to extensions that do not introduce new kinks.

\begin{proposition}
\label{prop:stable-extensions-f}
Let $w \in S$ be a stable word.
Then $f(\Sigma(w)) \subseteq \Sigma(f(w))$ and $f(\Tilde \Sigma(w, p)) \subseteq \Tilde \Sigma(f(w), p+1)$ for all $p \in \ZZ$.
If $w$ is not of the form $0^{\alpha} (10)^{n} 1^{\beta}$ with $\alpha, \beta \in \{0,1\}$ and $n \in \NN$, then both inclusions are equalities. 
\end{proposition}

In the case that $w$ is of the form $0^{\alpha} (10)^{n} 1^{\beta}$, the result holds ``up to parity'': for each $u = a f(w) b \in \Sigma(f(w))$ with $u_i = 0$ whenever $i \not\equiv |a| + \alpha \bmod 2$, we have $u \in f(\Sigma(w))$.

We also explain here the error in \cite{KUR03}.
Page 15 of the article contains the claim that rule 18 is $\Sigma_p$-permuting, which in our notation means $f(\Sigma(w)) \subseteq \Sigma(f(w))$ for every word $w \in \{0,1\}^*$ of length at least 3.
This may be false if $w$ is not stable.
For example, take $w = 0011$, for which $f(w) = 10$, and consider $101 \in \Sigma(f(w))$.

\begin{proof}
Let $a w b \in \Sigma(w)$.
We claim that $f(a w b) \in \Sigma(f(w))$, that is, all kinks of $f(a w b)$ are contained in $f(w)$.
Consider any subword $u$ of $f(a w b)$ such that $f(u) = 1 0^{2k} 1$ is a kink.
There are two possibilities:
\begin{enumerate}
    \item
    $u = 1 0^{2k+2} 1$ is itself a kink, and wider than $f(u)$.
    If $f(u)$ is not contained in $f(w)$, then $u$ is not contained in $w$, which contradicts $a w b \in \Sigma(w)$.
    \item
    $u = 0 0 v 0 0$ for some word $v$ that begins and ends with $1$, does not contain an occurrence of $0 0$, and contains at least one occurrence of $1 1$ (which must be within $w$).
    If $f(u)$ is not contained in $f(w)$, then $u$ is not contained in $w$, which contradicts the stability of $w$.
\end{enumerate}

Suppose then that $w$ is not of the form $0^{\alpha} (10)^{n} 1^{\beta}$, and let $a f(w) b \in \Sigma(f(w))$.
We claim that some $c w d \in \Sigma(w)$ satisfies $f(c w d) = a f(w) b$.
It suffices to consider $|a| = 1$ and $|b| = 0$; the general result follows by iteration and symmetry.
There are a few cases:
\begin{enumerate}
    \item If $a = 0$ and $w$ does not begin with $0 1$, we can choose $c = 0$.
    \item If $a = 0$ and $w$ begins with $0 1$, we can choose $c = 1$.
    \item If $a = 1$ and $w$ begins with $0 0$, then the assumption $a f(w) b \in \Sigma(f(w))$ implies that either $w = 0^{|w|}$ or $w$ begins with $0^{2n+1} 1$ for some $n \geq 1$.
    We can then choose $c = 1$.
    \item If $a = 1$ and $w$ begins with $0 1$, we can choose $c = 0$.
    \item Suppose $a = 1$ and $w$ begins with $1$.
    By stability and the assumption that $w$ is not of the form $0^{\alpha} (10)^{n} 1^{\beta}$, it must begin with $1 (0 1)^n 0 0$ for some $n \in \NN$.
    But then $a f(w)$ begins with $1 0^{2n} 1$, contradicting $a f(w) b \in \Sigma(f(w))$.
    Hence this case is impossible.
\end{enumerate}

The claims about $\tilde \Sigma$ follow easily.
\end{proof}

From this and Proposition \ref{prop:rule90} we get the following.

\begin{corollary}
    \label{cor:unique-extensions}
    Let $w \in S$ be a stable word of length at least 2 and not of the form $0^\alpha (10)^n 1^\beta$ for $\alpha, \beta \in \{0,1\}$.
    For each $u = a f(w) b \in \Sigma(f(w))$ there exist unique $a' \in \{0,1\}^{|a|}$ and $b' \in \{0,1\}^{|b|}$ with $f(a' w b') = u$.
\end{corollary}

\begin{corollary}\label{prsurj}
$f(\Sigma) = \Sigma$ and $f(\Tilde \Sigma) = \Tilde \Sigma$.
\end{corollary}





Finally, we make use of the following result to bring down the number of kinks in an initial configuration.

\begin{theorem}[\cite{JEN90}] \label{thJEN}
Let $x \in \{0,1\}^\ZZ$ be a configuration with a finite number of $1$-symbols.
Then there exists $n \in \NN$ such that $f^n(x)$ contains at most one kink.
\end{theorem}

Note that because of Proposition \ref{prop:no-new-kinks-stable}, the configuration $f^n(x)$ contains a kink if and only if $x$ has an odd number of kinks.

\begin{corollary} \label{coJEN}
Let $w \in \{0,1\}^*$ be a word with an even number of kinks.
Then there exists $n \in \NN$ such that $f^n(\Sigma(0^n \cdot w \cdot 0^n)) \subseteq \Sigma$ and $f^n( \Tilde \Sigma(0^n \cdot w \cdot 0^n), p) \subseteq \Tilde \Sigma$ for all $p \in \ZZ$.
\end{corollary}

\section{The generic limit set of rule 18}

In this section we characterize the set of words with at most two kinks that occur in the generic limit set of rule 18.
In principle, the cases of zero or one kink (Theorems \ref{thm:no-kinks} and \ref{thm:one-kink}) are subsumed by the case of two kinks (Theorem \ref{thm:two-kinks}), since all words of at most one kink occur as subwords of two-kink words in the generic limit set.
We present proofs for these cases as well in order to showcase our techniques in a simpler context.
In the case of zero kinks, the idea is simple: use Theorem \ref{thJEN} to erase all existing kinks, then use the permutivity of $f_{18}$ on kinkless configurations to produce an arbitrary kinkless word.

\subsection{Kinkless words}

We begin with the easiest case of kinkless words.

\begin{theorem}
\label{thm:no-kinks}
Every word with no kinks occurs in $\gls(f_{18})$.
\end{theorem}

\begin{proof}
Let $w \in \Sigma$ be a kinkless word. We consider the seed $\epsilon$ at position $i = 0$. Let $q_1, q_2 \in \{0,1\}^*$ be arbitrary words. Our goal is to show that $f^k([q_1 q_2]_{-|q_1|})$ intersects $[w]_0$ for infinitely many $k$.

There exists a word $r \in \{0,1\}^*$ such that $q_1 q_2  r$ contains an even number of kinks. According to Corollary \ref{coJEN}, there exists $n_0 \in \NN$ such that for all $n \geq n_0$ and all configurations $x \in \Tilde \Sigma(0^n \cdot q_1 q_2 r \cdot 0^n, -|q_1|-n)$, the image $f^n(x)$ contains no kinks. In the following, we call $u = 0^{2n+2} \cdot q_1 q_2 r \cdot 0^{2n+2}$.

For each $m \leq n$, the word $f^m(u)$ begins and ends in $0 0$, so it is stable and not of the form $0^\alpha (1 0)^k 1^\beta$ for $\alpha, \beta \in \{0,1\}$ and $k \in \NN$.
By iterating Proposition \ref{prop:stable-extensions-f}, we have
\begin{align*}
    \tilde \Sigma(f^n(u), -|q_1|-n-2) = {} & f(\tilde \Sigma(f^{n-1}(u), -|q_1|-n-3)) \\
    {} = {} & f^2(\tilde \Sigma(f^{n-2}(u), -|q_1|-n-4)) \\
    {} \vdots {} & \\
    {} = {} & f^n(\tilde \Sigma(u, -|q_1|-2n-2)).
\end{align*}

The word $f^n(u)$, and hence each configuration in $\tilde \Sigma(f^n(u), -|q_1|-n-2)$, contains no kinks.
Suppose that $|q_1|, |q_2|$ are both even and all $1$-symbols in $f^n(u)$ and $w$ occur in even positions (the other cases are handled similarly).
Then $f^{|q_2r|+n+2}(u) w \in \Sigma(f^{|q_2r|+n+2}(u))$.
Again by iterating Proposition \ref{prop:stable-extensions-f} (and the remark following it), we obtain $\tilde \Sigma(f^{|q_2r|+2n+2}(u), -|q_1 q_2 r|-2n-2) = f^{|q2r|+n+2}(\tilde \Sigma(f^n(u), -|q_1|-2n-2))$.
Hence, as long as $q_1$ and $q_2$ are long enough (as we can choose them to be), there exists a word $v$ of length $|w|$ such that $f^n(u) v$ contains no kinks and $f^{|q_2 r|+n+2}(f^n(u) v)$ ends in $w$.
In total, we have
\[
\tilde \Sigma(w, 0) \subseteq f^{|q_2 r|+n+2}(\tilde \Sigma(f^n(u) v, -|q_1|-n-2)) \subseteq f^{|q_2 r|+2n+2}(\tilde \Sigma(u, -|q_1|-2n-2)),
\]
which shows that $[w]_0$ intersects $f^k([q_1 q_2]_{-|q_1|})$ for some $k \geq n$.
Since $n$ can be arbitrarily large, Lemma \ref{lemma:combchar} gives that $w$ occurs in $\gls(f_{18})$.
\end{proof}

\subsection{Words with one kink}

In this section, we show that all words with exactly one kink occur in the generic limit set of rule 18.
The idea is to again erase all existing kinks except one, then direct it to a specific position, and use permutivity to control the kinkless region around it.

\begin{lemma}
\label{lem:one-kink-to-unstable}
Let $w$ be a one kink word.
Then there exists $n \in \NN$ such that $f^n(w)$ is a left unstable or right unstable word.
\end{lemma}

\begin{proof}
We proceed by induction on $|w|$.
If $|w| = 2$ or $|w| = 3$, then $w$ must contain $11$ as a prefix or suffix, so it is already unstable.
Else, if $w$ is stable, it is easy to see that $f(w)$ is also a one kink word. By the induction hypothesis there exists some $n'$ such that $f^{n'}(f(w))$ is left or right unstable.
Hence the property is true for $n = n'+1$.
\end{proof}

\begin{definition}
A \emph{left kink word} is a finite word that begins with a kink and contains no other kinks.
The set of left kink words is denoted $L \subset \{0,1\}^*$.
\end{definition}

The following lemma tells us that $11$ can be produced from any word in $L$.
As any one-kink word can be produced in one step from $11$, it follows from this lemma that any one kink word can be produced from any word of $L$.

\begin{lemma} \label{prreduc}
For all left kink words $w \in L$, there exists $n \in \NN$ such that $\Sigma(11) \subseteq f^n(\Sigma(w))$.
\end{lemma}

\begin{proof}
We prove the result by induction on the size of $w$.
There are a few cases:
\begin{enumerate}
    \item If $|w| = 2$, then $w = 11$ and the result is immediate. Otherwise $|w| \geq 3$, and we suppose that the property holds for all shorter words.
    \item If $w$ begins with $10^{2k}1$ for some $k > 0$, then $f(w) \in L$ and $|f(w)| < |w|$.
    Then $\Sigma(11) \subseteq f^n(\Sigma(f(w)))$ for some $n \in \NN$ by the induction hypothesis, and since $w \in S$, Proposition \ref{prop:stable-extensions-f} gives $f^n(\Sigma(f(w))) \subseteq f^{n+1}(\Sigma(w))$.
    \item If $w$ is of the form $11 (01)^k 0$ for some $k \in \NN$, then we consider $w' = 00 w 0 \in \Sigma(w)$.
    We have $f^{k+2}(w') = 11$ and $\Sigma(f^{k+2}(w')) = f^{k+2}(\Sigma(w'))$ by Proposition \ref{prop:stable-extensions-f}, and we can choose $n = k+2$.
    \item If $w$ is of the form $11 (01)^k$ and some $k \in \NN$, then we consider $w' = 00 w 00 \in \Sigma(w)$, and conclude as in the previous case.
    \item If $w$ is of the form $11 (01)^k 00 u$ for some word $u$ and $k \in \NN$, then we consider $w' = 00 w$.
    We have $f(w') = 1 0^{2(k+1)} 1 v$ for some word $v$.
    Then we treat $f(w')$ as in case 2 and conclude.
\end{enumerate}
\end{proof}

The following lemma is a direct consequence from the previous one, it shows that any one kink word can be produced at any position from any other one kink word at position 0.

\begin{lemma}\label{lm1k}
Let $w$ and $w'$ be two words containing one kink.
Then for all $i \in \ZZ$, there exists $n \in \NN$ such that $\Tilde \Sigma(w', 0) \subseteq f^n(\Tilde \Sigma(w, i))$.
\end{lemma}

The idea is to let the kink of $w$ evolve into a $11$ whose context we can freely choose.
Then we can easily produce another $11$ at any desired position on a later time step.
This $11$ can then be transformed into $w'$ in a single step.

\begin{proof}
Let $k \in \NN$ be the smallest number for which $f^k(w)$ is unstable.
We may assume that it is left unstable.
By Proposition \ref{prop:no-new-kinks-stable}, $f^k(w)$ contains one kink, which must be $11$.
Then $f(00 \cdot f^k(w))$ is a left kink word.

According to Lemma \ref{prreduc}, there exists $m \in \NN$ such that $\Sigma(11) \subseteq f^m(\Sigma(f(00\cdot f^k(w))))$, and Proposition \ref{prop:stable-extensions-f} gives $f^m(\Sigma(f(00\cdot f^k(w)))) \subseteq f^{k + m+1}(\Sigma(w))$.
We can deduce that there exists a position $p \in \ZZ$ such that $\Tilde \Sigma(11, p) \subseteq f^{k + m+1}(\Tilde \Sigma(w, i))$.

Since $w'$ is a one kink word, it has the form $a 1 0^{2\ell} 1 b$ for some $a, b \in \{0,1\}^*$ and $\ell \in \NN$.
Suppose that $p \leq |a|$, the other case being similar.
Consider $w_0 = 0011 (01)^{|a|-p+1}00$.
We have $f^{|a|-p + 3}(w_0) = 11$, and all the intermediate steps are stable, hence Proposition \ref{prop:stable-extensions-f} gives $\Tilde \Sigma(11, |a|+1) = f^{|a|-p + 3}(\Tilde \Sigma(w_0, p-2)) \subseteq f^{|a|-p + 3}(\Tilde \Sigma(11, p))$. 

The word $u = 0011 (01)^{\ell} 00$ is stable and $f(u) = 1 0^{2\ell} 1$, so we obtain $\Tilde \Sigma(w', 0) \subseteq \tilde \Sigma(1 0^{2\ell} 1, |a|) \subseteq f(\tilde \Sigma(u'), |a|-1) \subseteq f(\tilde \Sigma(11, |a|+1))$ by Proposition \ref{prop:stable-extensions-f}.
Putting everything together, we have $\tilde \Sigma(w', 0) \subseteq f^{k+m+|a|-p+5}(\tilde \Sigma(w, i))$, as claimed.
\end{proof}

\begin{theorem}
\label{thm:one-kink}
Every one-kink word occurs in $\gls(f_{18})$.
\end{theorem}

\begin{proof}
Let $w \in \{0,1\}^*$ be a one-kink word. We consider the seed $\epsilon$. Let $q_1, q_2 \in \{0,1\}^*$ be words, and choose $r \in \{0,1\}^*$ and $n_0 \leq n \in \NN$ as in the proof of Theorem \ref{thm:no-kinks}, except that $q_1 \cdot q_2 \cdot r$ should contain an odd number of kinks and every configuration in $f^n(\tilde \Sigma(0^n q_1 q_2 r 0^n, -|q_1|-n))$ should contain a single king. Call $u = 0^{2n+2} \cdot q_1 \cdot \epsilon \cdot q_2 \cdot r \cdot 0^{2n+2}$.

For all $m \leq n$, $f^m(u)$ is a stable one-kink word, so that $\Tilde \Sigma(f^n(u), -|q_1|-n-2) = f^n(\Tilde \Sigma(u, -|q_1|-2n-2))$.
According to Lemma \ref{lm1k}, there exists $k \in \NN$ such that $\Tilde \Sigma(w, 0) \subseteq f^k(\Tilde \Sigma(f^n(u), -|q_1|-n-2)) = f^{n+k}(\Tilde \Sigma(u, -|q_1|-2n-2))$.
This means that $[w]$ intersects $f^{n+k}([q_1 q_2]_{-|q_1|})$.
Since $n$ can be arbitrarily large, Lemma \ref{lemma:combchar} implies that $w$ occurs in $\gls(f)$.
%
%
%
%
\end{proof}

\subsection{Words with two kinks}

In this section, we determine the set $P$ of words with exactly two kinks that appear in the generic limit set of rule $18$. The idea is to erase all existing kinks except two, then bring the two kinks close to each other in order to produce the word $1101001$ and then to show that all words of $P$ can be produced from a two-kink configuration that contains $1101001$ at any position.

Denote the reversal of a word $w \in \{0,1\}^*$ by $w^R$.
For a set $L \subseteq \{0,1\}^*$, we denote $L^R = \{w^R \mid w \in L\}$.

\begin{definition}
We call $B \subset \{0,1\}^*$ the set of two-kink words that begin with $11$, end with a kink (which can overlap with the prefix $11$), and are not of the form $11 \cdot (01)^k \cdot 1$ for $k \in \NN$.
\end{definition}

This definition is useful as $B$ is a set of words containing two kinks that we can prevent from colliding.
The following lemma states that we can direct the two kinks of a word of $B$ so that they can be brought close and produce the word $1101001$.

\begin{lemma}\label{lm2k1}
Let $w \in B \cup B^R$ be a word.
Then there exists $n \in \NN$ such that $\Sigma(1101001) \subseteq f^n(\Sigma(w))$.
\end{lemma}

The proof of this lemma contains a lot of different cases, however, the cases are quite natural and the proof of each case is easy. The main idea of the proof is that by forbidding the kinks to move away from each other, it is possible to bring them eventually closer. When the two kinks are on the verge of colliding, it is possible to avoid the collision and to instead make them closer to each other as illustrated in figure \ref{fig:kink-avoid-merge}.

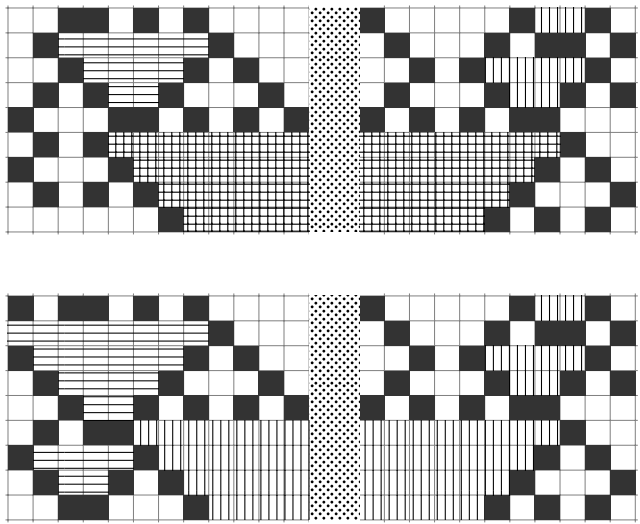
\begin{figure}[htp]
\centering
    \begin{tabular}{cc}
    \begin{subfigure}[b]{.5\linewidth}
    \begin{tikzpicture}[scale=0.33]
\draw[step=1cm, gray, very thin](-0.1,-0.1)grid(25.1,9.1);
\fill[black!80](2,8)rectangle(4,9);
\fill[black!80](5,8)rectangle(6,9);
\fill[black!80](7,8)rectangle(8,9);
\fill[black!80](14,8)rectangle(15,9);
\fill[black!80](20,8)rectangle(21,9);
\fill[black!80](23,8)rectangle(24,9);
\fill[black!80](1,7)rectangle(2,8);
\fill[black!80](8,7)rectangle(9,8);
\fill[black!80](15,7)rectangle(16,8);
\fill[black!80](19,7)rectangle(20,8);
\fill[black!80](21,7)rectangle(23,8);
\fill[black!80](24,7)rectangle(25,8);
\fill[black!80](2,6)rectangle(3,7);
\fill[black!80](7,6)rectangle(8,7);
\fill[black!80](9,6)rectangle(10,7);
\fill[black!80](16,6)rectangle(17,7);
\fill[black!80](18,6)rectangle(19,7);
\fill[black!80](23,6)rectangle(24,7);
\fill[black!80](1,5)rectangle(2,6);
\fill[black!80](3,5)rectangle(4,6);
\fill[black!80](6,5)rectangle(7,6);
\fill[black!80](10,5)rectangle(11,6);
\fill[black!80](15,5)rectangle(16,6);
\fill[black!80](19,5)rectangle(20,6);
\fill[black!80](22,5)rectangle(23,6);
\fill[black!80](24,5)rectangle(25,6);
\fill[black!80](0,4)rectangle(1,5);
\fill[black!80](4,4)rectangle(6,5);
\fill[black!80](7,4)rectangle(8,5);
\fill[black!80](9,4)rectangle(10,5);
\fill[black!80](11,4)rectangle(12,5);
\fill[black!80](14,4)rectangle(15,5);
\fill[black!80](16,4)rectangle(17,5);
\fill[black!80](18,4)rectangle(19,5);
\fill[black!80](20,4)rectangle(22,5);
\fill[black!80](1,3)rectangle(2,4);
\fill[black!80](3,3)rectangle(4,4);
\fill[black!80](22,3)rectangle(23,4);
\fill[black!80](0,2)rectangle(1,3);
\fill[black!80](4,2)rectangle(5,3);
\fill[black!80](21,2)rectangle(22,3);
\fill[black!80](23,2)rectangle(24,3);
\fill[black!80](1,1)rectangle(2,2);
\fill[black!80](3,1)rectangle(4,2);
\fill[black!80](5,1)rectangle(6,2);
\fill[black!80](20,1)rectangle(21,2);
\fill[black!80](24,1)rectangle(25,2);
\fill[black!80](6,0)rectangle(7,1);
\fill[black!80](19,0)rectangle(20,1);
\fill[black!80](21,0)rectangle(22,1);
\fill[black!80](23,0)rectangle(24,1);

\fill[white](12,-1)rectangle(14,10);
\fill[pattern=crosshatch dots, pattern color=black] (12,0) rectangle (14,9);

\fill[pattern=horizontal lines, pattern color=black] (2,7) rectangle (8,8);
\fill[pattern=horizontal lines, pattern color=black] (3,6) rectangle (7,7);
\fill[pattern=horizontal lines, pattern color=black] (4,5) rectangle (6,6);

\fill[pattern=vertical lines, pattern color=black] (21,8) rectangle (23,9);
\fill[pattern=vertical lines, pattern color=black] (19,6) rectangle (23,7);
\fill[pattern=vertical lines, pattern color=black] (20,5) rectangle (22,6);

\fill[pattern=grid, pattern color=black] (14,3) rectangle (22,4);
\fill[pattern=grid, pattern color=black] (14,2) rectangle (21,3);
\fill[pattern=grid, pattern color=black] (14,1) rectangle (20,2);
\fill[pattern=grid, pattern color=black] (14,0) rectangle (19,1);

\fill[pattern=grid, pattern color=black] (4,3) rectangle (12,4);
\fill[pattern=grid, pattern color=black] (5,2) rectangle (12,3);
\fill[pattern=grid, pattern color=black] (6,1) rectangle (12,2);
\fill[pattern=grid, pattern color=black] (7,0) rectangle (12,1);
\end{tikzpicture}
    \end{subfigure}
    \\
    \begin{subfigure}[b]{.5\linewidth}
    \begin{tikzpicture}[scale=0.33]
\draw[step=1cm, gray, very thin](-0.1,-0.1)grid(25.1,9.1);
\fill[black!80](0,8)rectangle(1,9);
\fill[black!80](2,8)rectangle(4,9);
\fill[black!80](5,8)rectangle(6,9);
\fill[black!80](7,8)rectangle(8,9);
\fill[black!80](14,8)rectangle(15,9);
\fill[black!80](20,8)rectangle(21,9);
\fill[black!80](23,8)rectangle(24,9);
\fill[black!80](0,6)rectangle(1,7);
\fill[black!80](8,7)rectangle(9,8);
\fill[black!80](15,7)rectangle(16,8);
\fill[black!80](19,7)rectangle(20,8);
\fill[black!80](21,7)rectangle(23,8);
\fill[black!80](24,7)rectangle(25,8);
\fill[black!80](7,6)rectangle(8,7);
\fill[black!80](9,6)rectangle(10,7);
\fill[black!80](16,6)rectangle(17,7);
\fill[black!80](18,6)rectangle(19,7);
\fill[black!80](23,6)rectangle(24,7);
\fill[black!80](1,5)rectangle(2,6);
\fill[black!80](6,5)rectangle(7,6);
\fill[black!80](10,5)rectangle(11,6);
\fill[black!80](15,5)rectangle(16,6);
\fill[black!80](19,5)rectangle(20,6);
\fill[black!80](22,5)rectangle(23,6);
\fill[black!80](24,5)rectangle(25,6);
\fill[black!80](2,4)rectangle(3,5);
\fill[black!80](5,4)rectangle(6,5);
\fill[black!80](7,4)rectangle(8,5);
\fill[black!80](9,4)rectangle(10,5);
\fill[black!80](11,4)rectangle(12,5);
\fill[black!80](14,4)rectangle(15,5);
\fill[black!80](16,4)rectangle(17,5);
\fill[black!80](18,4)rectangle(19,5);
\fill[black!80](20,4)rectangle(22,5);
\fill[black!80](1,3)rectangle(2,4);
\fill[black!80](3,3)rectangle(5,4);
\fill[black!80](22,3)rectangle(23,4);
\fill[black!80](0,2)rectangle(1,3);
\fill[black!80](5,2)rectangle(6,3);
\fill[black!80](21,2)rectangle(22,3);
\fill[black!80](23,2)rectangle(24,3);
\fill[black!80](1,1)rectangle(2,2);
\fill[black!80](4,1)rectangle(5,2);
\fill[black!80](6,1)rectangle(7,2);
\fill[black!80](20,1)rectangle(21,2);
\fill[black!80](24,1)rectangle(25,2);
\fill[black!80](2,0)rectangle(4,1);
\fill[black!80](7,0)rectangle(8,1);
\fill[black!80](19,0)rectangle(20,1);
\fill[black!80](21,0)rectangle(22,1);
\fill[black!80](23,0)rectangle(24,1);

\fill[white](12,-1)rectangle(14,10);
\fill[pattern=crosshatch dots, pattern color=black] (12,0) rectangle (14,9);

\fill[pattern=horizontal lines, pattern color=black] (0,7) rectangle (8,8);
\fill[pattern=horizontal lines, pattern color=black] (1,6) rectangle (7,7);
\fill[pattern=horizontal lines, pattern color=black] (2,5) rectangle (6,6);
\fill[pattern=horizontal lines, pattern color=black] (3,4) rectangle (5,5);
\fill[pattern=horizontal lines, pattern color=black] (4,5) rectangle (4,4);
\fill[pattern=horizontal lines, pattern color=black] (1,2) rectangle (5,3);
\fill[pattern=horizontal lines, pattern color=black] (2,1) rectangle (4,2);

\fill[pattern=vertical lines, pattern color=black] (21,8) rectangle (23,9);
\fill[pattern=vertical lines, pattern color=black] (19,6) rectangle (23,7);
\fill[pattern=vertical lines, pattern color=black] (20,5) rectangle (22,6);
\fill[pattern=vertical lines, pattern color=black] (14,3) rectangle (22,4);
\fill[pattern=vertical lines, pattern color=black] (14,2) rectangle (21,3);
\fill[pattern=vertical lines, pattern color=black] (14,1) rectangle (20,2);
\fill[pattern=vertical lines, pattern color=black] (14,0) rectangle (19,1);

\fill[pattern=vertical lines, pattern color=black] (5,3) rectangle (12,4);
\fill[pattern=vertical lines, pattern color=black] (6,2) rectangle (12,3);
\fill[pattern=vertical lines, pattern color=black] (7,1) rectangle (12,2);
\fill[pattern=vertical lines, pattern color=black] (8,0) rectangle (12,1);
\end{tikzpicture}
 
    \end{subfigure}
    \end{tabular}
    \caption{Two kinks collide and annihilate each other, but it is possible to avoid collision by making the leftmost kink longer.}
    \label{fig:kink-avoid-merge}
\end{figure}

\begin{proof}
First we show that it suffices to prove the result for $w \in B$.
Namely, then for $w \in B^R$ we have $\Sigma(1001011) \subseteq f^n(\Sigma(w))$ for some $n \in \NN$ by symmetry, and a direct computation shows that $f^3(10 \cdot 1001011 \cdot 0100) = 1101001$, so that $\Sigma(1101001) \subseteq f^{n+3}(\Sigma(w))$.

Let then $w \in B$ be arbitrary.
We proceed by induction on the length of $w$.

\paragraph{Case 0: $|w| \leq 7$.}
There are only a few possibilities for $w$.
If $w = 1101001$, the result is immediate.
If $w = 11001$, then $\Sigma(1001011) \subseteq f^3(\Sigma(w))$ as shown in Figure \ref{fig:2k_forw_11001} and $\Sigma(1101001) \subseteq f^3(\Sigma(1001011))$ as shown in Figure \ref{fig:2k_forw_switch}, hence $\Sigma(1101001) \subseteq f^6(\Sigma(w))$.
If $w = 1100001$, then $\Sigma(1101001) \subseteq f^2(\Sigma(1001011)) \subseteq f^5(\Sigma(w))$ as shown in Figure \ref{fig:2k_forw_1100001}.
If $w = 1100011$, then $\Sigma(1101001) \subseteq f^8(\Sigma(000010 \cdot w \cdot 0101000100)) \subseteq f^8(\Sigma(w))$ as shown in Figure \ref{fig:2k_forw_000}.

\begin{figure}[htp]
    \centering
    \begin{tikzpicture}[scale=0.6]
\draw[step=1cm, gray, very thin](-0.1,-0.1)grid(13.1,4.1);

\fill[black!80](2,3)rectangle(3,4);
\fill[black!80](4,3)rectangle(5,4);
\fill[black!80](6,3)rectangle(8,4);
\fill[black!80](10,3)rectangle(11,4);
\fill[black!80](12,3)rectangle(13,4);
\fill[black!80](1,2)rectangle(2,3);
\fill[black!80](8,2)rectangle(10,3);
\fill[black!80](2,1)rectangle(3,2);
\fill[black!80](7,1)rectangle(8,2);
\fill[black!80](10,1)rectangle(11,2);
\fill[black!80](3,0)rectangle(4,1);
\fill[black!80](6,0)rectangle(7,1);
\fill[black!80](8,0)rectangle(10,1);

\fill[pattern=north west lines, pattern color=black!80] (0,2) rectangle (1,3);
\fill[pattern=north west lines, pattern color=black!80] (0,1) rectangle (2,2);
\fill[pattern=north west lines, pattern color=black!80] (0,0) rectangle (3,1);

\fill[pattern=north west lines, pattern color=black!80] (12,2) rectangle (13,3);
\fill[pattern=north west lines, pattern color=black!80] (11,1) rectangle (13,2);
\fill[pattern=north west lines, pattern color=black!80] (10,0) rectangle (13,1);
\end{tikzpicture}
    \caption{We can see here that $w = 0010101100101$ is stable, $f(w)$ and $f^2(w)$ are stable and $f^3(w) = 1001011$. We can deduce that $\Sigma(1001011) = f^3(\Sigma(w)) \subseteq f^3(\Sigma(1001011))$.}
    \label{fig:2k_forw_11001}
\end{figure}
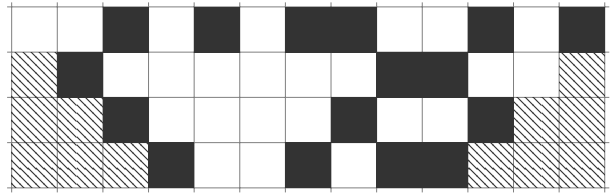

\begin{figure}[htp]
    \centering
    \begin{tikzpicture}[scale=0.6]
\begin{scope}[yscale=1,xscale=-1]
\draw[step=1cm, gray, very thin](-0.1,-0.1)grid(13.1,4.1);

\fill[black!80](2,3)rectangle(3,4);
\fill[black!80](4,3)rectangle(6,4);
\fill[black!80](7,3)rectangle(8,4);
\fill[black!80](10,3)rectangle(11,4);
\fill[black!80](12,3)rectangle(13,4);
\fill[black!80](1,2)rectangle(2,3);
\fill[black!80](8,2)rectangle(10,3);
\fill[black!80](2,1)rectangle(3,2);
\fill[black!80](7,1)rectangle(8,2);
\fill[black!80](10,1)rectangle(11,2);
\fill[black!80](3,0)rectangle(4,1);
\fill[black!80](6,0)rectangle(7,1);
\fill[black!80](8,0)rectangle(10,1);

\fill[pattern=north west lines, pattern color=black!80] (0,2) rectangle (1,3);
\fill[pattern=north west lines, pattern color=black!80] (0,1) rectangle (2,2);
\fill[pattern=north west lines, pattern color=black!80] (0,0) rectangle (3,1);

\fill[pattern=north west lines, pattern color=black!80] (12,2) rectangle (13,3);
\fill[pattern=north west lines, pattern color=black!80] (11,1) rectangle (13,2);
\fill[pattern=north west lines, pattern color=black!80] (10,0) rectangle (13,1);
\end{scope}
\end{tikzpicture}
    \caption{We can see here that $w = 1010010110100$ is stable, $f(w)$ and $f^2(w)$ are stable and $f^3(w) = 1101001$. We can deduce that $\Sigma(1101001) = f^3(\Sigma(w)) \subseteq f^3(\Sigma(1001011))$.}
    \label{fig:2k_forw_switch}
\end{figure}
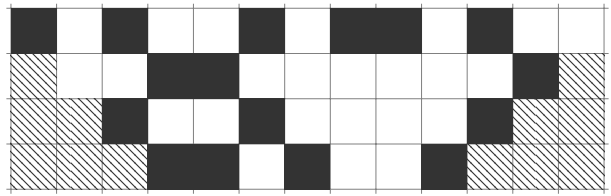

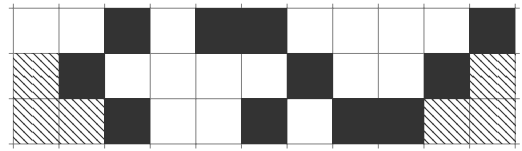
\begin{figure}[htp]
    \centering
    \begin{tikzpicture}[scale=0.6]
\draw[step=1cm, gray, very thin](-0.1,-0.1)grid(11.1,3.1);

\fill[black!80](2,2)rectangle(3,3);
\fill[black!80](4,2)rectangle(6,3);
\fill[black!80](10,2)rectangle(11,3);
\fill[black!80](1,1)rectangle(2,2);
\fill[black!80](6,1)rectangle(7,2);
\fill[black!80](9,1)rectangle(10,2);
\fill[black!80](2,0)rectangle(3,1);
\fill[black!80](5,0)rectangle(6,1);
\fill[black!80](7,0)rectangle(9,1);

\fill[pattern=north west lines, pattern color=black!80] (0,1) rectangle (1,2);
\fill[pattern=north west lines, pattern color=black!80] (0,0) rectangle (2,1);

\fill[pattern=north west lines, pattern color=black!80] (10,1) rectangle (11,2);
\fill[pattern=north west lines, pattern color=black!80] (9,0) rectangle (11,1);

\end{tikzpicture}
    \caption{We can see here that $w = 00101100001$ is stable, $f(w)$ is stable and $f^2(w) = 1001011$. We can deduce that $\Sigma(1101001) \subseteq f^2(\Sigma(1001011)) = f^5(\Sigma(w)) \subseteq f^5(\Sigma(1100001))$.}
    \label{fig:2k_forw_1100001}
\end{figure}

\begin{figure}[htp]
    \centering
    \begin{tikzpicture}[scale=0.4]
\draw[step=1cm, gray, very thin](-0.1,-0.1)grid(23.1,9.1);
\fill[black!80](4,8)rectangle(5,9);
\fill[black!80](6,8)rectangle(8,9);
\fill[black!80](11,8)rectangle(13,9);
\fill[black!80](14,8)rectangle(15,9);
\fill[black!80](16,8)rectangle(17,9);
\fill[black!80](20,8)rectangle(21,9);
\fill[black!80](3,7)rectangle(4,8);
\fill[black!80](8,7)rectangle(9,8);
\fill[black!80](10,7)rectangle(11,8);
\fill[black!80](17,7)rectangle(18,8);
\fill[black!80](19,7)rectangle(20,8);
\fill[black!80](21,7)rectangle(22,8);
\fill[black!80](2,6)rectangle(3,7);
\fill[black!80](4,6)rectangle(5,7);
\fill[black!80](7,6)rectangle(8,7);
\fill[black!80](11,6)rectangle(12,7);
\fill[black!80](16,6)rectangle(17,7);
\fill[black!80](5,5)rectangle(7,6);
\fill[black!80](8,5)rectangle(9,6);
\fill[black!80](10,5)rectangle(11,6);
\fill[black!80](12,5)rectangle(13,6);
\fill[black!80](15,5)rectangle(16,6);
\fill[black!80](17,5)rectangle(18,6);
\fill[black!80](4,4)rectangle(5,5);
\fill[black!80](13,4)rectangle(15,5);
\fill[black!80](18,4)rectangle(19,5);
\fill[black!80](5,3)rectangle(6,4);
\fill[black!80](12,3)rectangle(13,4);
\fill[black!80](15,3)rectangle(16,4);
\fill[black!80](17,3)rectangle(18,4);
\fill[black!80](6,2)rectangle(7,3);
\fill[black!80](11,2)rectangle(12,3);
\fill[black!80](13,2)rectangle(15,3);
\fill[black!80](7,1)rectangle(8,2);
\fill[black!80](10,1)rectangle(11,2);
\fill[black!80](15,1)rectangle(16,2);
\fill[black!80](8,0)rectangle(10,1);
\fill[black!80](11,0)rectangle(12,1);
\fill[black!80](14,0)rectangle(15,1);

\fill[pattern=north west lines, pattern color=black!80] (0,8) rectangle (0,9);
\fill[pattern=north west lines, pattern color=black!80] (0,7) rectangle (1,8);
\fill[pattern=north west lines, pattern color=black!80] (0,6) rectangle (2,7);
\fill[pattern=north west lines, pattern color=black!80] (0,5) rectangle (3,6);
\fill[pattern=north west lines, pattern color=black!80] (0,4) rectangle (4,5);
\fill[pattern=north west lines, pattern color=black!80] (0,3) rectangle (5,4);
\fill[pattern=north west lines, pattern color=black!80] (0,2) rectangle (6,3);
\fill[pattern=north west lines, pattern color=black!80] (0,1) rectangle (7,2);
\fill[pattern=north west lines, pattern color=black!80] (0,0) rectangle (8,1);

\fill[pattern=north west lines, pattern color=black!80] (22,7) rectangle (23,8);
\fill[pattern=north west lines, pattern color=black!80] (21,6) rectangle (23,7);
\fill[pattern=north west lines, pattern color=black!80] (20,5) rectangle (23,6);
\fill[pattern=north west lines, pattern color=black!80] (19,4) rectangle (23,5);
\fill[pattern=north west lines, pattern color=black!80] (18,3) rectangle (23,4);
\fill[pattern=north west lines, pattern color=black!80] (17,2) rectangle (23,3);
\fill[pattern=north west lines, pattern color=black!80] (16,1) rectangle (23,2);
\fill[pattern=north west lines, pattern color=black!80] (15,0) rectangle (23,1);
\end{tikzpicture}
    \caption{We can see here that $w = 00001011000110101000100$ is stable, $f(w), f^2(w), \ldots, f^7(w)$ are stable and $f^8(w) = 1101001$. We can deduce that $\Sigma(1101001) = f^8(\Sigma(w)) \subseteq f^8(\Sigma(1100011))$.}
    \label{fig:2k_forw_000}
\end{figure}

For the rest of the proof we suppose $|w| > 7$, and all shorter words of $B \cup B^R$ satisfy the claim.

\paragraph{Case 1: $w = 11 0^{2k} 1$ for some $k \in \NN$.}
Then $k \geq 3$. We consider $w_0 = 00 \cdot w$. Then, $f^2(w_0) = 1101  0^{2(k-2)} 1$. Evidently $f^2(w_0) \in B$, $|f^2(w_0)| < |w|$ and $\Sigma(f^2(w_0)) \subseteq f^2(\Sigma(w))$, so we can conclude by induction.

\paragraph{Case 2: $w = 1101 0^{2k} 1$ for some $k \in \NN$.}
Then $k \geq 2$. We consider $w_0 = 00 \cdot w$. If $k = 2$, then $f^2(w_0) = 1001011$, which we treated in Case 0. If $k = 3$, then $f^3(w_0) = 1100011$, which we already treated as well. If $k \geq 4$, then $f^3(w_0) = 110001 0^{2(k-3)} 1 \in B$ and $|f^3(w_0)| < |w|$. We can then conclude by induction.

\paragraph{Case 3: $w = 11 (01)^k 0^{2\ell} 1$ for some $k, \ell \geq 1$.}
We consider $w_0 = 00 \cdot w$. Then, $f(w_0) = 1 0^{2(k+1)} 1 0^{2(\ell-1)} 1$, so that for $m = \min(k+1, \ell-1)$ the word $w' = f^{m+1}(w_0)$ begins and/or ends with $11$. If $\ell = 1$, then $|w'| = |w|$, and we can handle $w'$ as in Case 1. Otherwise $|w'| < |w|$.
If $w' \in B \cup B^R$, then we can conclude by induction.
If $w' \notin B \cup B^R$, then $\ell = k+2$ and $w' = 11 (01)^{2(k+1)} 1$.
Then we instead consider $w_1 = 0010 \cdot w$, for which $f^{k+1}(w_1) = 1001 \cdot u \cdot 11$ for some word $u$ of length $2(k+1)$.
This word is in $B^R$ and is shorter than $w$, so we conclude by induction.

If we are not in the above cases, then the gap between the two kinks is longer than one cell and contains an occurrence of $00$.
Hence there exist $k, \ell \in \NN$ and a word $u$ such that $w = 11 (01)^k 00 \cdot u \cdot 1 0^{2\ell} 1$.
The remaining cases constrain the values of $k$ and $\ell$.

\paragraph{Case 4: $\ell \geq 2$.}
We again consider $w_0 = 00 \cdot w$, and split into subcases depending on the value of $k$.
If $k = 0$, we consider $w_0' = f^2(w_0)$, which begins with $11$ and ends with $10^{2(\ell-2)}1$.
If $w_0' \in B$, we can conclude by induction as $|w_0'| < |w|$. Else, $\ell = 2$ and $w_0' = 11 (01)^{(|w|-6)/2} 011$. Then we instead consider $v_0 = 0010 \cdot w \cdot 00$, which satisfies $f^3(v_0) = 11 (00)^{(|w|-3)/6} \cdot 1$.
We can now conclude using Case 1, as $|f^3(v_0)| = |w|$.

If $0 < k < \ell-2$, then $f^{k+2}(w_0) \in B$ and $|w_0| < |w|$, so we can conclude using induction.

If $k = \ell-2$, then consider $w' = f^\ell(w_0)$.
If $w' \in B$, we have $|w'| < |w|$ and we conclude by induction.
If $w' \notin B$, we instead consider $w_1 = 0010 \cdot w$. Then $f^{\ell+1}(w_1) \in B^R$ and $|f^{\ell+1}(w_1)| < |w|$, so we conclude using induction. 

If $k > \ell-2$, then $f^\ell(w_0) \in B^R$ and $|f^\ell(w_0)| < |w|$. We conclude by induction.

\paragraph{Case 5: $\ell \leq 1$ and $k \geq 1$.}
If $\ell = 0$, then we consider $w_0 = 00 \cdot w \cdot 00$. Now $f(w_0)$ begins with $1 0^{2(k+2)} 1$ and ends with $1 0^{2 m} 1$ for some $m \geq 1$.
We proceed as in Case 3, using $w_1 = 0010 \cdot w \cdot 00$ instead of $w_0$ is needed.
If $\ell = 1$, then $f(w_0)$ is in $B^R$, has the same length as $w$, and belongs to Case 3, which we already solved.

\paragraph{Case 6: $\ell = 1$ and $k = 0$.}
We consider $w_0 = 00 \cdot w$, and denote $w_0' = f(w_0)^R$, which begins with $11$ and ends with $1 001$.
If $w_0'$ belongs to one of the already solved cases 0--5, then we are done.
Otherwise, $w_0' \in B$ belongs to Case 6.
Then we repeat this process: for each $i \in \NN$, as long as $w_i'$ belongs to Case 6, define $w_{i+1} = 00 \cdot w_i'$ and $w_{i+1}' = f(w_{i+1})^R = 11 \cdot u_{i+1} \cdot 1001$.
Suppose that we reach $w_{|w|}$ this way.
We can now go backward in time as shown in figure \ref{fig:2k_back_11_1001}: for each $j \geq 0$, the length-$j$ prefix and suffix of $u_{|w|-j}$ are determined locally.
In the end, we determine that $w = 1 \cdot (100010)^{(|w| - 5)/6} \cdot 1001$. 

\begin{figure}[htp]
    \centering
    \begin{tikzpicture}[scale=0.5]
\draw[step=1cm, gray, very thin](-0.1,-2.1)grid(22.1,7.1);
\draw[black, very thick](0,-2)--(0, 7);
\draw[black, very thick](0,7)--(4, 7);
\draw[black, very thick](4,-2)--(4, 7);
\draw[black, very thick](0,-2)--(4, -2);
\fill[black!80](1,-2)rectangle(3,-1);
\fill[black!80](1,0)rectangle(3,1);
\fill[black!80](1,2)rectangle(3,3);
\fill[black!80](1,4)rectangle(3,5);
\fill[black!80](1,6)rectangle(3,7);
\fill[black!80](0,-1)rectangle(1,0);
\fill[black!80](0,1)rectangle(1,2);
\fill[black!80](0,3)rectangle(1,4);
\fill[black!80](0,5)rectangle(1,6);
\fill[black!80](3,-1)rectangle(4,0);
\fill[black!80](3,1)rectangle(4,2);
\fill[black!80](3,3)rectangle(4,4);
\fill[black!80](3,5)rectangle(4,6);
\fill[black!80](5,1)rectangle(6,2);
\fill[black!80](5,3)rectangle(6,4);
\fill[black!80](5,5)rectangle(6,6);
\fill[black!80](6,2)rectangle(7,3);
\fill[black!80](6,4)rectangle(7,5);
\fill[black!80](6,6)rectangle(7,7);
\fill[black!80](8,4)rectangle(9,5);
\fill[black!80](8,6)rectangle(9,7);
\fill[black!80](9,5)rectangle(10,6);

\fill[white](10,-3)rectangle(12,8);
\fill[pattern=crosshatch dots, pattern color=black] (10,-2) rectangle (12,7);

\draw[black, very thick](18,-2)--(18, 7);
\draw[black, very thick](18,7)--(22, 7);
\draw[black, very thick](22,-2)--(22, 7);
\draw[black, very thick](18,-2)--(22, -2);
\fill[black!80](19,-1)rectangle(21,0);
\fill[black!80](19,1)rectangle(21,2);
\fill[black!80](19,3)rectangle(21,4);
\fill[black!80](19,5)rectangle(21,6);
\fill[black!80](18,-2)rectangle(19,-1);
\fill[black!80](18,0)rectangle(19,1);
\fill[black!80](18,2)rectangle(19,3);
\fill[black!80](18,4)rectangle(19,5);
\fill[black!80](18,6)rectangle(19,7);
\fill[black!80](21,-2)rectangle(22,-1);
\fill[black!80](21,0)rectangle(22,1);
\fill[black!80](21,2)rectangle(22,3);
\fill[black!80](21,4)rectangle(22,5);
\fill[black!80](21,6)rectangle(22,7);
\fill[black!80](16,0)rectangle(17,1);
\fill[black!80](16,2)rectangle(17,3);
\fill[black!80](16,4)rectangle(17,5);
\fill[black!80](16,6)rectangle(17,7);
\fill[black!80](15,1)rectangle(16,2);
\fill[black!80](15,3)rectangle(16,4);
\fill[black!80](15,5)rectangle(16,6);
\fill[black!80](13,3)rectangle(14,4);
\fill[black!80](13,5)rectangle(14,6);
\fill[black!80](12,4)rectangle(13,5);
\fill[black!80](12,6)rectangle(13,7);

\fill[pattern=north west lines, pattern color=black!80] (4,-2) rectangle (10,-1);
\fill[pattern=north west lines, pattern color=black!80] (5,-1) rectangle (10,0);
\fill[pattern=north west lines, pattern color=black!80] (6,0) rectangle (10,1);
\fill[pattern=north west lines, pattern color=black!80] (7,1) rectangle (10,2);
\fill[pattern=north west lines, pattern color=black!80] (8,2) rectangle (10,3);
\fill[pattern=north west lines, pattern color=black!80] (9,3) rectangle (10,4);
\fill[pattern=north west lines, pattern color=black!80] (12,-2) rectangle (18,-1);
\fill[pattern=north west lines, pattern color=black!80] (12,-1) rectangle (17,0);
\fill[pattern=north west lines, pattern color=black!80] (12,0) rectangle (16,1);
\fill[pattern=north west lines, pattern color=black!80] (12,1) rectangle (15,2);
\fill[pattern=north west lines, pattern color=black!80] (12,2) rectangle (14,3);
\fill[pattern=north west lines, pattern color=black!80] (12,3) rectangle (13,4);
\end{tikzpicture}
    \caption{At each step of the backward computation, we force the four left and four right symbols to be either $0110$ or $1001$, this in turn fixes the neighbors symbols to some value. After enough steps, the whole word is forced.}
    \label{fig:2k_back_11_1001}
\end{figure}
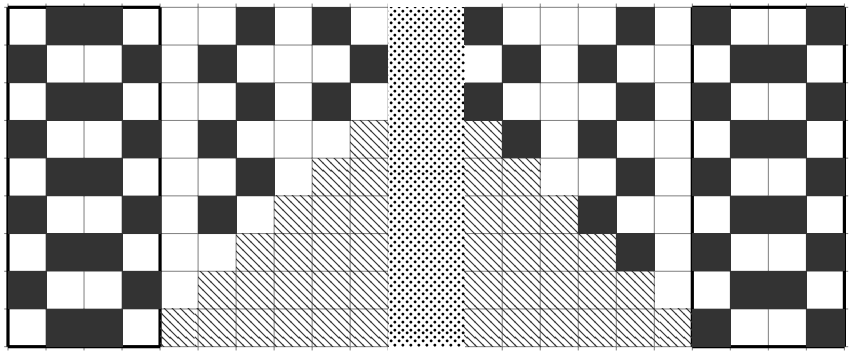

We now consider $v_0 = 0010 \cdot w \cdot 01$ instead of $w_0$. Then $v_0' = f^3(v_0) = 110 \cdot (101000)^{(|w|-5)/6} \cdot 11$, and $v_1' = f^4(00 \cdot v_0' \cdot 0001)$ verifies $v_1' \in B$ and $|v_1'| < |w|$. We can then conclude using induction.

\paragraph{Case 7: $\ell = k = 0$.}
We consider $w_0 = 00  \cdot w \cdot 00$ and $w_0' = f^2(w_0)$.
If $w_0'$ belongs to one of the already solved cases, then since $|w_0'| = |w|$, we are done.
If $w_0'$ belongs to Case 7, we again repeat the process: for each $i \in \NN$, as long as $w_i'$ belongs to Case 7, define $w_{i+1} = 00 \cdot w_i' \cdot 00$ and $w_{i+1}' = f^2(w_{i+1}) = 11 \cdot u_{i+1} \cdot 11$.
In the same way as in Case 6, if we reach $w_{|w|}$ in this way, we go backward in time as shown in Figure \ref{fig:2k_back_11_11} to fully determine that $w = 11000 \cdot (101000)^{(|w| - 7)/6} \cdot 11$.

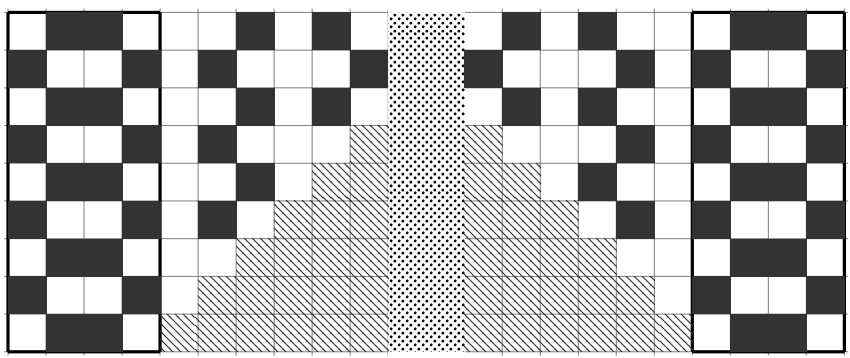
\begin{figure}[htp]
    \centering
    \begin{tikzpicture}[scale=0.5]
\draw[step=1cm, gray, very thin](-0.1,-2.1)grid(22.1,7.1);
\draw[black, very thick](0,-2)--(0, 7);
\draw[black, very thick](0,7)--(4, 7);
\draw[black, very thick](4,-2)--(4, 7);
\draw[black, very thick](0,-2)--(4, -2);
\fill[black!80](1,-2)rectangle(3,-1);
\fill[black!80](1,0)rectangle(3,1);
\fill[black!80](1,2)rectangle(3,3);
\fill[black!80](1,4)rectangle(3,5);
\fill[black!80](1,6)rectangle(3,7);
\fill[black!80](0,-1)rectangle(1,0);
\fill[black!80](0,1)rectangle(1,2);
\fill[black!80](0,3)rectangle(1,4);
\fill[black!80](0,5)rectangle(1,6);
\fill[black!80](3,-1)rectangle(4,0);
\fill[black!80](3,1)rectangle(4,2);
\fill[black!80](3,3)rectangle(4,4);
\fill[black!80](3,5)rectangle(4,6);
\fill[black!80](5,1)rectangle(6,2);
\fill[black!80](5,3)rectangle(6,4);
\fill[black!80](5,5)rectangle(6,6);
\fill[black!80](6,2)rectangle(7,3);
\fill[black!80](6,4)rectangle(7,5);
\fill[black!80](6,6)rectangle(7,7);
\fill[black!80](8,4)rectangle(9,5);
\fill[black!80](8,6)rectangle(9,7);
\fill[black!80](9,5)rectangle(10,6);

\fill[white](10,-3)rectangle(12,8);
\fill[pattern=crosshatch dots, pattern color=black] (10,-2) rectangle (12,7);

\draw[black, very thick](18,-2)--(18, 7);
\draw[black, very thick](18,7)--(22, 7);
\draw[black, very thick](22,-2)--(22, 7);
\draw[black, very thick](18,-2)--(22, -2);
\fill[black!80](19,-2)rectangle(21,-1);
\fill[black!80](19,0)rectangle(21,1);
\fill[black!80](19,2)rectangle(21,3);
\fill[black!80](19,4)rectangle(21,5);
\fill[black!80](19,6)rectangle(21,7);
\fill[black!80](18,-1)rectangle(19,0);
\fill[black!80](18,1)rectangle(19,2);
\fill[black!80](18,3)rectangle(19,4);
\fill[black!80](18,5)rectangle(19,6);
\fill[black!80](21,-1)rectangle(22,0);
\fill[black!80](21,1)rectangle(22,2);
\fill[black!80](21,3)rectangle(22,4);
\fill[black!80](21,5)rectangle(22,6);
\fill[black!80](16,1)rectangle(17,2);
\fill[black!80](16,3)rectangle(17,4);
\fill[black!80](16,5)rectangle(17,6);
\fill[black!80](15,2)rectangle(16,3);
\fill[black!80](15,4)rectangle(16,5);
\fill[black!80](15,6)rectangle(16,7);
\fill[black!80](13,4)rectangle(14,5);
\fill[black!80](13,6)rectangle(14,7);
\fill[black!80](12,5)rectangle(13,6);

\fill[pattern=north west lines, pattern color=black!80] (4,-2) rectangle (10,-1);
\fill[pattern=north west lines, pattern color=black!80] (5,-1) rectangle (10,0);
\fill[pattern=north west lines, pattern color=black!80] (6,0) rectangle (10,1);
\fill[pattern=north west lines, pattern color=black!80] (7,1) rectangle (10,2);
\fill[pattern=north west lines, pattern color=black!80] (8,2) rectangle (10,3);
\fill[pattern=north west lines, pattern color=black!80] (9,3) rectangle (10,4);
\fill[pattern=north west lines, pattern color=black!80] (12,-2) rectangle (18,-1);
\fill[pattern=north west lines, pattern color=black!80] (12,-1) rectangle (17,0);
\fill[pattern=north west lines, pattern color=black!80] (12,0) rectangle (16,1);
\fill[pattern=north west lines, pattern color=black!80] (12,1) rectangle (15,2);
\fill[pattern=north west lines, pattern color=black!80] (12,2) rectangle (14,3);
\fill[pattern=north west lines, pattern color=black!80] (12,3) rectangle (13,4);
\end{tikzpicture}
    \caption{In a same way as in figure \ref{fig:2k_back_11_1001}, after enough steps of backward computation, the whole word is forced.}
    \label{fig:2k_back_11_11}
\end{figure}

We now consider $v_0 = 0010 \cdot w \cdot 0001$. Then $v_0' = f^3(v_0) = 110101 \cdot (000101)^{(|w|-7)/6} \cdot 001$ and $v_1' = f^4(00 \cdot v_0' \cdot 0100) = 1100 \cdot (000101)^{(|w|-7)/6} \cdot 001$. We are now in Case 6, which was already solved, with $|w| = |v_1'|$.
This finishes the proof.
\end{proof}

\begin{lemma}\label{lm2kmov}
Let $i, j \in \ZZ^2$ be two positions, then there exists $n \in \NN$ such that $\Tilde \Sigma(1101001, i) \subseteq f^n(\Tilde \Sigma(1101001, j)$. 
\end{lemma}

\begin{proof}
We proceed by induction on $|i-j|$.
If $i = j$, we can choose $n = 0$.
If $i < j$, then $\tilde \Sigma(1101001, i+1) \subseteq f^5(\Tilde \Sigma(1101001, j))$ as shown in figure \ref{fig:2k_forw_leftmov} and we can conclude using induction.
If $i> j$, then, $\Tilde \Sigma(1101001, i-1) \subseteq f^3(\Tilde \Sigma(1101001, j)$ as shown in figure \ref{fig:2k_forw_rightmov} and we can conclude using induction.
\end{proof}

    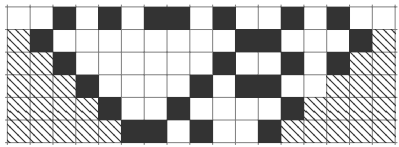
\begin{figure}[htp]
      \centering
    \begin{tikzpicture}[scale=0.3]
\draw[step=1cm, gray, very thin](-0.1,-0.1)grid(17.1,6.1);
\fill[black!80](2,5)rectangle(3,6);
\fill[black!80](4,5)rectangle(5,6);
\fill[black!80](6,5)rectangle(8,6);
\fill[black!80](9,5)rectangle(10,6);
\fill[black!80](12,5)rectangle(13,6);
\fill[black!80](14,5)rectangle(15,6);
\fill[black!80](1,4)rectangle(2,5);
\fill[black!80](10,4)rectangle(12,5);
\fill[black!80](15,4)rectangle(16,5);
\fill[black!80](2,3)rectangle(3,4);
\fill[black!80](9,3)rectangle(10,4);
\fill[black!80](12,3)rectangle(13,4);
\fill[black!80](14,3)rectangle(15,4);
\fill[black!80](3,2)rectangle(4,3);
\fill[black!80](8,2)rectangle(9,3);
\fill[black!80](10,2)rectangle(12,3);
\fill[black!80](4,1)rectangle(5,2);
\fill[black!80](7,1)rectangle(8,2);
\fill[black!80](12,1)rectangle(13,2);
\fill[black!80](5,0)rectangle(7,1);
\fill[black!80](8,0)rectangle(9,1);
\fill[black!80](11,0)rectangle(12,1);

\fill[pattern=north west lines, pattern color=black!80] (0,4) rectangle (1,5);
\fill[pattern=north west lines, pattern color=black!80] (0,3) rectangle (2,4);
\fill[pattern=north west lines, pattern color=black!80] (0,2) rectangle (3,3);
\fill[pattern=north west lines, pattern color=black!80] (0,1) rectangle (4,2);
\fill[pattern=north west lines, pattern color=black!80] (0,0) rectangle (5,1);

\fill[pattern=north west lines, pattern color=black!80] (16,4) rectangle (17,5);
\fill[pattern=north west lines, pattern color=black!80] (15,3) rectangle (17,4);
\fill[pattern=north west lines, pattern color=black!80] (14,2) rectangle (17,3);
\fill[pattern=north west lines, pattern color=black!80] (13,1) rectangle (17,2);
\fill[pattern=north west lines, pattern color=black!80] (12,0) rectangle (17,1);
\end{tikzpicture}
    \caption{Moving a $1101001$ one step to the left.}
    \label{fig:2k_forw_leftmov}
    \end{figure}

    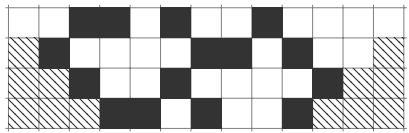
\begin{figure}[htp]
    \centering
    \begin{tikzpicture}[scale=0.4]
\draw[step=1cm, gray, very thin](-0.1,-0.1)grid(13.1,4.1);

\fill[black!80](2,3)rectangle(4,4);
\fill[black!80](5,3)rectangle(6,4);
\fill[black!80](8,3)rectangle(9,4);
\fill[black!80](1,2)rectangle(2,3);
\fill[black!80](6,2)rectangle(8,3);
\fill[black!80](9,2)rectangle(10,3);
\fill[black!80](2,1)rectangle(3,2);
\fill[black!80](5,1)rectangle(6,2);
\fill[black!80](10,1)rectangle(11,2);
\fill[black!80](3,0)rectangle(5,1);
\fill[black!80](6,0)rectangle(7,1);
\fill[black!80](9,0)rectangle(10,1);

\fill[pattern=north west lines, pattern color=black!80] (0,2) rectangle (1,3);
\fill[pattern=north west lines, pattern color=black!80] (0,1) rectangle (2,2);
\fill[pattern=north west lines, pattern color=black!80] (0,0) rectangle (3,1);

\fill[pattern=north west lines, pattern color=black!80] (12,2) rectangle (13,3);
\fill[pattern=north west lines, pattern color=black!80] (11,1) rectangle (13,2);
\fill[pattern=north west lines, pattern color=black!80] (10,0) rectangle (13,1);
\end{tikzpicture}
    \caption{Moving a $1101001$ one step to the right.}
    \label{fig:2k_forw_rightmov}
    \end{figure}

The following definition provides us with a set $P$ of words that can be produced from $1101001$, and as we will see later, these words constitute exactly the words of the generic limit set that have two kinks.

\begin{definition}
Let $w$ be a word that contains exactly two kinks.
We call $b(w)$ the smallest subword of $w$ that contains two kinks. 
If $b(w) = 1 0^{2k} 1 \cdot u \cdot 1 0^{2\ell} 1$ for some $k, \ell \in \NN$ and $u \in \{0,1\}^*$, then we call $\delta(w) = u$.
If $b(w)$ is of the form $10^{2k} 1 0^{2\ell} 1$, then we call $\delta(w) = \epsilon$.
\end{definition}

\begin{definition}
\label{def:P}
We define $P \subset \{0,1\}^*$ as the set of those two-kink words $w$ that satisfy all of the following:
\begin{itemize}
    \item $b(w) \neq 111$.
    \item If $b(w) = 11 v 11$, then $v$ contains an even number of $1$s.
    \item For all $k \in \NN$, $b(w) \neq 1 (100010)^k 1001$ and $b(w) \neq 1001 (010001)^k 1$.
\end{itemize}
\end{definition}

Our goal is to prove that $P$ coincides with the set of two-kink words of the generic limit set of $f$.
The first two conditions are necessary, since the words $111$ and $11v11$ with $v$ containing an odd number of $1$s have no preimages under $f$.
Namely, the only preimage of $11$ is $1001$, which already forbids $111$.
On the other hand, can prove by induction on the length of $w$ that if $w$ contains an even number of $1$s, then all preimages of $11w$ end in $01$, $10$ or $11$, and otherwise they all end in $00$.
Then $11v11$ has no preimage, as it would have to end in both $1001$ and $00ab$ for some $a, b \in \{0,1\}$.
Similarly, every two-kink word $11w11$ where $w$ contains an even number of $1$s has a unique two-kink preimage (and possibly other preimages with more kinks): if $w = 0^{2n_0-1} 1 0^{2n_1-1} 1 \cdots 1 0^{2n_{2k}-1}$, that preimage is $100 (10)^{n_0} 0^{2n_1} (10)^{n_2} 0^{2n_3} \cdots (10)^{n_{2k}} 01$.

We will now proceed backwards in time and try to find for every word in $P$ a parent that contains $1101001$.
Just as in Lemma \ref{lm2k1}, the proof is a combination of an induction and a case analysis.
There are quite many cases, but none of them is very difficult.
The main idea is that we are first trying to close the gap between the two kinks in the word, and then shorten the kinks.

\begin{lemma}\label{lm2k2}
Let $w \in P$. Then there exists $n \in \NN$ and a position $p \in \ZZ$ such that $\Tilde \Sigma(w, p) \subseteq f^n(\Tilde \Sigma(1101001, 0))$.
\end{lemma}

\begin{proof}
We first handle the ``trivial'' cases.

\paragraph{Case 0.}
If $1101001$ is a subword of $w$, then the property is verified.
If $1001011$ is a subword of $w$, then $\Sigma(1001011) \subseteq f^3(\Sigma(1101001))$ as shown in Figure \ref{fig:2k_forw_switch}.
Hence there exists a position $p$ such that $\Tilde \Sigma(w, p) \subseteq f^3(\Tilde \Sigma(1101001, 0))$.

We proceed by induction on the length of $\delta(w)$.
Let $w$ be a word such that $|\delta(w)| = 0$. Then there exist some words $\alpha, \beta$ and some integers $k, \ell$ such that $w = \alpha 1 0^{2k} 1 0^{2\ell} 1 \beta$.

\paragraph{Case 1: $k = 0$ or $\ell = 0$.}
If $k = 0$ (the case $\ell = 0$ is symmetrical), then $w \in P$ implies $\ell \geq 2$. According to Corollary \ref{cor:unique-extensions}, there exists a unique word $\alpha'$ and a unique word $\beta'$ such that $w' = \alpha' \cdot 1001011 (01)^{\ell-2} 00 \cdot \beta'$ contains two kinks and $f(w') = w$. Since $w' \in \Sigma(1001011)$ and $w'$ is stable, we have $\tilde \Sigma(w, 1) = f(\tilde \Sigma(w',0))$. The word $w'$ belongs to Case 0, so $\tilde \Sigma(w',p) \subseteq f^n(\tilde \Sigma(1101001, 0))$ for some $p \in \ZZ$.
Then $\tilde \Sigma(w,p+1) \subseteq f^{n+1}(\tilde \Sigma(1101001, 0))$ by Proposition \ref{prop:stable-extensions-f}.
In the remainder of the proof, we implicitly apply this argument whenever we fall to an earlier case or use the induction hypothesis.

\paragraph{Case 2: $k \geq 1$ and $\ell \geq 1$.}
Now there exist unique words $\alpha'$ and $\beta'$ such that $w' = \alpha' \cdot 1 0^{2(k+1)} 11 (01)^{\ell-1} 00\cdot \beta'$ contains two kinks and $f(w') = w$. Since $w'$ is stable, we have $\Sigma(w) = f(\Sigma(w'))$. Hence, we can conclude by Case 1.

In the following, let $w \in P$ be a word with $|\delta(w)| > 0$. We now suppose the property holds for all words $w' \in P$ with $|\delta(w')| < |\delta(w)|$. Let $\alpha$ and $\beta$ be words and $k, \ell$ be integers such that $w = \alpha \cdot 1 0^{2k} 1 \cdot \delta(w) \cdot 1 0^{2\ell} 1 \cdot \beta$.
Because of Case 0, we may assume $k \leq \ell$.

\paragraph{Case 3: $|\delta(w)| = 1$.}
Then we must have $\delta(w) = 0$.
There exist unique words $\alpha'$ and $\beta'$ such that $w' = \alpha' \cdot 10^{2(k+1)}10^{2(l+1)}1 \cdot \beta'$ contains two kinks and $f(w') = w$. As $w'$ is stable, we can conclude using Case 2.

\paragraph{Case 4: $|\delta(w)| \geq 3$, $k= 0$ and $\ell=0$.}
As $w \in P$, the word $\delta(w)$ between the two kinks necessarily contains an even number of $1$s. According to the discussion following Definition \ref{def:P}, there exists a word of the form $w' = 1001 \delta 1001$ with two kinks such that $f(w') = w$. 
Now $w' \in P$ and $\delta(w') = \delta$ verifies $|\delta(w')| < |\delta(w)|$, so we can conclude by the induction hypothesis.
    
\paragraph{Case 5: $k = 0$ and $\ell = 1$.}
Now there exists a unique word $u'$ such that $1001 \cdot u'$ contains one kink and $f(1001 \cdot u') = 11 \cdot \delta(w)$.
    
If $u'$ is of the form $v' \cdot 10$ for some word $v'$, then there exist unique words $\alpha'$ and $\beta'$ such that $w' = \alpha' \cdot 1001 \cdot v' \cdot 100001 \cdot \beta'$ contains two kinks and $f(w') = w$. Now $\delta(w') = v'$, hence $|\delta(w')| < |\delta(w)|$ and we can conclude by induction.
    
Otherwise $u'$ is of the form $v' \cdot 00$. There exists a word $u''$ such that $1001 \cdot u''$ contains one kink and $f(1001 \cdot u'')^R = u' \cdot 11$. We now consider $u''$: if it ends in $10$, we conclude as above, otherwise we continue.

 We suppose that we are able to repeat the process an infinite number of times: we can define words $\alpha^{(k)}$, $\beta^{(k)}$, $u^{(k)}$ for all $k \in \NN$ such that the word $w^{(k)} = \alpha^{(k)} \cdot 1001 \cdot u^{(k+1)} \cdot 11 \cdot \beta^{(k)}$ contains two kinks, $w^{(0)} = w$, and $f(w^{(k+1)})^R = w^{(k)}$. As in the proof of Case 6 of Lemma \ref{lm2k1}, and in Figure \ref{fig:2k_back_11_1001}, we can go backwards in time to deduce prefixes and suffixes of the $u^{(k)}$ and eventually determine that $w^{(|w|)} = 1 \cdot (100010)^{(|w| - 4)/6} \cdot 1001$. Then by direct computation we have $w = 1 \cdot (100010)^{(|w| - 4)/6} \cdot 1001$ as well. We then proved that $w \notin P$, a contradiction.
    
\paragraph{Case 6: $k = 0$ and $\ell>1$.}
Again, there exists a unique word $u'$ such that $1001 \cdot u'$ contains one kink and $f(1001 \cdot u') = 11 \cdot \delta(w)$. 
    
If $u'$ is of the form $v' \cdot 10$ for some word $v'$, then there exist unique words $\alpha'$ and $\beta'$ such that $w' = \alpha' \cdot 1 00 1 \cdot v' \cdot 1 0^{2(\ell+1)}1 \cdot \beta'$ contains two kinks and $f(w') = w$. Now $w' \in P$, and $\delta(w') = v'$ implying $|\delta(w')| < |\delta(w)|$, so we conclude by induction.
    
Otherwise $u'$ ends in $00$. There exist unique words $\alpha'$ and $\beta'$ such that $w' = \alpha' \cdot 1001 \cdot u' \cdot 11 (01)^{\ell-1} 00 \cdot \beta'$ contains two kinks and $f(w') = w$.
If $w' \in P$, then it is in Case 5 with $|\delta(w')| = |\delta(w)|$, and we can conclude.

Otherwise $w' \notin P$, that is, there exists $m \in \NN$ such that $u' \cdot 1 = (010001)^m$.
This implies $w = \alpha \cdot 1 (100010)^m 10^{2\ell}1 \cdot \beta$, and we call it the \emph{difficult subcase}.
We define
\[
u'' =
\begin{cases}
001011 (000000010101)^{m/2} 00001 & \text{if $m$ is even,} \\
1000000(101010000000)^{(m-1)/2} 1010100001 & \text{if $m$ is odd.}
\end{cases}
\]
In both cases it can be checked by a direct computation that $f^2(u'') = 1001 \cdot u' \cdot 1011$.
Then there are unique words $\alpha''$ and $\beta''$ such that $w'' = \alpha'' \cdot u'' \cdot \beta''$ has two kinks and satisfies $f^3(w'') = w$.
If $m$ is even, then $|\delta(w'')| = |\delta(w)|$ and $w''$ belongs to Case 6, and by its form it cannot be in the difficult subcase.
If $m$ is odd, then $|\delta(w'')| < |\delta(w)|$.
Hence we can conclude.

\paragraph{Case 7: $k \geq 1$ and $\ell \geq 1$.}
Again there exists a unique word $u'$ such that $1 0^{2(k+1)} 1 \cdot u'$ contains one kink and $f(10^{2(k+1)}1 \cdot u') = 10^{2k}1 \cdot \delta(w)$.
    
If $u'$ is of the form $v' \cdot 10$ for some word $v'$, then there exist unique words $\alpha'$ and $\beta'$ such that $w' = \alpha' \cdot 10^{2(k+1)}1 \cdot v' \cdot 10^{2(l+1)}1 \cdot \beta'$ contains two kinks and $f(w') = w$. We have $\delta(w') = v'$, hence $|\delta(w')| < |\delta(w)|$ and we conclude by induction.
    
Else, $u'$ ends in $00$. Then there exist unique words $\alpha'$ and $\beta'$ with $w' = \alpha' \cdot 10^{2(k+1)} 1 \cdot u' \cdot 11 (01)^{\ell-1} \cdot \beta'$ and $f(w') = w$. We have $w' \in P$ and $|\delta(w')| = |\delta(w)|$.
Moreover, $w'$ belongs to the already handled Case 6.
This concludes the proof.
\end{proof}

\begin{theorem}
\label{thm:two-kinks}
Let $w \in \{0,1\}^*$ be a word that contains two kinks.
Then $w$ occurs in $\gls(f_{18})$ if and only if $w \in P$.
\end{theorem}

\begin{proof}
First, let $w \notin P$ be a word with two kinks.
If $w$ has a subword $111$ or a subword of the form $11 u 11$ with $u$ containing an odd number of $1$s, then $w$ has no parent, and hence it does not occur in $\gls(f_{18})$.

Suppose then that $w$ has a subword of the form $u = 1 (100010)^k 1001$ for some $k \in \NN$.
We call $u' = 1001 (010001)^k 1$.
These have the property that $f^{-1}([u]_0) \cap f(\{0,1\}^*) \subseteq [u']_{-1}$ and $f^{-1}([u']_{-1}) \cap f(\{0,1\}^*) \subseteq [u]_0$.
We suppose for a contradiction that $w \in \gls(f_{18})$.
Then there exists a seed $s \in \{0,1\}^*$ and a position $p \in \ZZ$ such that for all $a, b \in \{0,1\}^*$, there are infinitely many $n \in \NN$ such that $f^n([asb]_{i-|a|}) \cap [w]_0 \neq \emptyset$.
We may assume that $p \leq 0$ and $|s|-p \geq |w|$.

There exists a word $q$ such that $s q$ contains an even number of kinks.
By Corollary \ref{coJEN}, there exists $\ell \in \NN$ such that $f^\ell(0^\ell s q 0^\ell)$ is kinkless.
Then we have $f^\ell(x)_{[0, |w|-1]} \in \Sigma$ for all $x \in [0^\ell s q 0^\ell]_{p-\ell}$.
However, there should exist $n \geq \ell$ and $x \in [0^\ell s q 0^\ell]_{p-\ell}$ with $f^n(x) \in [w]_0$.
This implies $f^n(x) \in [u]_j$ for some $0 \leq j < |w|-|u|$.
Then for all $m \leq n$, we have $f^{n-m}(x) \in [u]_j$ if $m$ is even, and $f^{n-m}(x) \in [u']_{j-1}$ if $m$ is odd.
When $m = n-\ell$, we obtain a contradiction with $f^\ell(x)_{[0, |w|-1]} \in \Sigma$, since $u$ and $u'$ are two-kink words.
Hence, $w$ does not occur in the generic limit set.

We now prove the other side of the implication.
Let $w \in P$ be a word. We consider the seed $\epsilon$ at position $i = 0$. Let $q_1, q_2 \in \{0,1\}^*$ be words. There exists a word $r \in \{0,1\}^*$ such that $q_1 \cdot q_2 \cdot r$ contains an even number of kinks. According to Corollary \ref{coJEN}, for all large enough $n \in \NN$, all words in $f^n(\Sigma(0^n \cdot q_1 q_2  r \cdot 0^n, -|q_1|-n))$ are kinkless. In the following, we call $s = 1 0^{2n} 1 0^{3n} \cdot q_1 q_2 r \cdot 0^{3n} 1 0^{2n} 1$.

We have $f^n(\Sigma(s)) = \Sigma(f^n(s))$ by Proposition \ref{prop:stable-extensions-f}.
Also, $f^n(s)$ is a two-kink word of the form $11 \cdot v \cdot 11$ where $v \neq (01)^{\lfloor v/2 \rfloor} 0$, so that $f^n(s) \in B$.
According to Lemma \ref{lm2k1}, there exists $n_0 \in \NN$ such that $\Sigma(1101001) \subseteq f^{n_0}(\Sigma(f^n(s))) = f^{n+n_0}(\Sigma(s))$.
Hence, there exists a position $j \in \ZZ$ such that $\Tilde \Sigma(1101001, j) \subseteq f^{n + n_0}(\Tilde \Sigma(s, - |q_1| - 5n - 2))$.

According to Lemma \ref{lm2k2}, there exists $n_1 \in \NN$ and a position $p \in \ZZ$ such that $\Tilde \Sigma(w, p) \subseteq f^{n_1}(\Tilde \Sigma(1101001, 0)$.
According to Lemma \ref{lm2kmov}, there exists $n_2 \in \NN$ such that $\Tilde \Sigma(1101001, -p) \subseteq f^{n_2}(\Tilde \Sigma(1101001, j))$. Hence $\Tilde \Sigma(w, 0) \subseteq f^{n_1 +n_2}(\Tilde \Sigma(1101001, j)$.
Combining all of this, we obtain
\[
\Tilde \Sigma(w, 0) \subseteq f^{n+n_0+n_1+n_2}(\Tilde \Sigma(s, - |q_1| - 5n - 2)) \subseteq f^{n+n_0+n_1+n_2}([q_1 q_2]_{-|q_1|}).
\]
As $n$ can be chosen arbitrarily large, this shows that $w$ occurs in $\gls(f)$.
\end{proof}

\section{Separation of limit sets}

In this section we show that the limit set, generic limit set, and $\mu$-limit set of $f_{18}$ are distinct, for any shift-invariant probability measure $\mu$.
In fact, there is a single word that occurs in $\gls(f_{18})$ but does not occur in any $\limset_\mu(f_{18})$.

\begin{theorem}
  Rule 18 satisfies $\gls(f_{18}) \subsetneq \limset(f_{18})$ and $\gls(f_{18}) \setminus \overline{\bigcup_{\mu \in \meas_\sigma(\{0,1\}^\ZZ)} \limset_\mu(f_{18})} \neq \emptyset$.
\end{theorem}

\begin{proof}
We first prove that $\gls(f_{18})$ is properly contained in the limit set.
The word $10011$ occurs in the latter, since the spatially periodic point $x = {}^\infty(1001)^\infty$ satisfies $f_{18}^2(x) = x$.
By Theorem \ref{thm:two-kinks}, it does not occur in $\gls(f_{18})$.

For the other claim, take any shift-invariant measure $\mu$, and for $n \in \NN$, denote by
\[ d_n = \sum_{k \in \NN} (f_{18}^n \mu)(1 0^{2k} 1) \]
the probability of having the left border of a kink at a given coordinate after $n$ iterations of $f_{18}$.
By Proposition \ref{prop:no-new-kinks-stable}, $f_{18}$ cannot create new kinks, and by shift-invariance of $\mu$ we then have $d_{n+1} \leq d_n$.
The probability of seeing the left border of a kink that will be destroyed on the next step is exactly $d_n - d_{n+1}$.

Consider the word $w = 001101100$.
Since $f_{18}(w) = 1000001$ results in the destruction of two kinks, we have $\sum_{n \in \NN} (f_{18}^n \mu)(w) \leq \sum_{n \in \NN} d_n - d_{n+1} = \lim_{n \to \infty} d_0 - d_n \leq 1$.
Hence $\lim_{n \to \infty} (f_{18}^n \mu)(w) = 0$, so $w$ does not occur in $\limset_\mu(f_{18})$.
By Theorem \ref{thm:two-kinks}, it occurs in $\gls(f_{18})$.
\end{proof}

Let $\mu$ be the uniform Bernoulli measure.
In \cite[Conjecture 3]{LIN84}, Lind conjectures that the probability $d_n$ in the above proof is approximately $(8 \pi D n)^{-1/2}$, where the ``diffusion coefficient'' $D \approx 1/2$.
The weaker claim that $d_n \to 0$ is equivalent to the condition that no kinks occur in the $\mu$-limit set $\limset_\mu(f_{18})$.
By \cite[Proposition 2.10]{BDS10}, the $\mu$-limit set of every cellular automaton $g$ is \emph{shift-recurrent}, meaning that if $w$ occurs in $\limset_\mu(g)$, then there exists a word $u$ such that $wuw$ occurs in $\limset_\mu(g)$ as well.
Hence, for the weaker claim it would be enough to prove that $\limset_\mu(f_{18}) \subseteq \gls(f_{18})$ and that for some $k \in \NN$, no word with $k$ kinks occurs in the generic limit set $\gls(f_{18})$.
We have showed that the latter claim is false for $k \leq 2$.

\bibliographystyle{plain}
\bibliography{bibl}

\end{document}